\documentclass[11pt]{amsart}

\usepackage{amssymb,amsthm,amsmath,hyperref}
\usepackage[alphabetic]{amsrefs}
\usepackage[shortlabels]{enumitem}

\usepackage[margin=1.5in]{geometry}

\newtheorem{theorem}{Theorem}[section]
\newtheorem{corollary}[theorem]{Corollary}
\newtheorem{lemma}[theorem]{Lemma}
\newtheorem{proposition}[theorem]{Proposition}
\newtheorem{definition}[theorem]{Definition}

\theoremstyle{remark}
\newtheorem{remark}[theorem]{Remark}
\newtheorem{example}[theorem]{Example}

\numberwithin{equation}{section}

\newcommand{\N}{\mathbb{N}}

\newcommand{\ZZ}{\mathcal{Z}}
\newcommand{\R}{\mathbb{R}}
\newcommand{\C}{\mathbb{C}}

\newcommand{\Text}{\mathrm{T}}
\newcommand{\QT}{\mathrm{QT}}

\newcommand{\Tb}{\mathrm{T}_1}

\newcommand{\Cu}{\mathrm{Cu}}

\newcommand{\sa}{\mathrm{sa}}
\newcommand{\Tr}{\mathrm{Tr}}
\newcommand{\her}{\mathrm{her}}

\newcommand{\Cstar}{\rm C\sp*}

\renewcommand{\epsilon}{\varepsilon}
\renewcommand{\leq}{\leqslant}
\renewcommand{\geq}{\geqslant}

\title{Sums of commutators in pure $\Cstar$-algebras}
\author{Ping Wong Ng}
\address{\hskip-\parindent
Department of Mathematics,
University of Louisiana at Lafayette,
Lafayette, USA.}
\email{png@louisiana.edu}

\author{Leonel Robert}
\address{\hskip-\parindent
Department of Mathematics,
University of Louisiana at Lafayette,
Lafayette, USA.}
\email{lrobert@louisiana.edu}


\begin{document}
\maketitle

\begin{abstract}
In a pure  $\Cstar$-algebra (i.e., one having suitable regularity properties in its Cuntz semigroup), any element on which all bounded traces vanish  is a sum of 7 commutators. 
\end{abstract}	

\section{Introduction} 
This paper is concerned with the problem of representing trace zero elements in a $\Cstar$-algebra as sums of commutators. This problem has a long history, going back to the result by Shoda that a matrix of zero trace is expressible as a single commutator (i.e., has the form $xy-yx$). In a general $\Cstar$-algebra, 
one can deduce from the Hahn-Banach theorem that the elements that vanish on every bounded trace belong to the norm closure of the linear span of the commutators. One can even arrange, by a result of Cuntz and Pedersen \cite{cuntz-pedersen},  for a series of commutators converging in norm to any given trace zero element. A problem that has occupied numerous authors (\cite{fack,marcoux06,thomsen,pop, ng1,robert-commutators}) is that of turning this infinite sum of commutators into a finite one.  Examples in \cite{pedersen-petersen} and more recently \cite{robert-commutators} show that this is not always possible; not even for simple nuclear $\Cstar$-algebras with a unique tracial state. Marcoux \cite{marcoux06}, continuing work of Fack \cite{fack} and Thomsen \cite{thomsen}, was the first to show that $\Cstar$-algebraic regularity properties, such as Blackadars's strict comparison of projections, could be used to obtain a positive answer. This idea has proven fruitful, and the present paper extends further the work in this direction. We prove our results in the setting of pure $\Cstar$-algebras; i.e., $\Cstar$-algebras
whose Cuntz semigroups   have certain algebraic regularity properties. The class of pure $\Cstar$-algebras   includes  all  $\mathcal Z$-stable  $\Cstar$-algebras (i.e., those  tensorially absorbing  the Jiang-Su algebra) and tensorially prime examples such as the reduced $\Cstar$-algebra of the free group in infinitely many generators.
%
%

Let us fix some notation: 
Let $A$ be a $\Cstar$-algebra. By a commutator in $A$ we understand an element of the form $xy-yx$; we denote it by $[x,y]$.  We denote by $[A,A]$ the linear span of the commutators of $A$. 

Following Winter \cite{winter}, we say that $A$ is pure
if its Cuntz semigroup has the properties of almost divisibility and almost unperforation.  The latter property is equivalent to the strict comparison of positive elements by lower semicontinuous 2-quasitraces (see Section \ref{strictcomparison}). 
We prove the following theorem:
\begin{theorem}\label{mainpure}
Let $A$ be a pure $\Cstar$-algebra  whose lower semicontinuous $[0,\infty]$-valued   2-quasitraces are traces. Then for every  $h\in \overline{[A,A]}$
	we have $h=\sum_{i=1}^7[x_i,y_i]$, with $x_i,y_i\in A$ such that $\|x_i\|\cdot \|y_i\|\leq C\|h\|$ for all $i$, and where $C$ is a universal constant.
\end{theorem}

A significant departure in this theorem from past results is that   the existence of a unit in the $\Cstar$-algebra is not assumed.  This brings new technical difficulties that can nevertheless be overcome.  The assumption of simplicity for the $\Cstar$-algebra, typically present in previous results on this question, has also been dropped.

Part of the motivation for this paper has been to investigate pure $\Cstar$-algebras for their own sake. Indeed, toward the proof of Theorem \ref{mainpure}, we establish a number of results on pure $\Cstar$-algebras of intrinsic interest. Pure $\Cstar$-algebras arise naturally in  the classification program for simple nuclear $\Cstar$-algebras and in various $\Cstar$-algebra constructions. For example, the tensor product of any  $\Cstar$-algebra with the Jiang-Su algebra $\mathcal Z$  is  pure. However,  while it is reasonable to  expect that many naturally occurring simple $\Cstar$-algebras are pure,  there is no evidence that $\mathcal Z$-stability  is a prevalent property beyond the realm of nuclear $\Cstar$-algebras.  
Theorem \ref{mainpure} applies to infinite  reduced free products which can be both nonexact and tensorially prime (see Example \ref{freeproducts}).

On the way to proving Theorem \ref{mainpure}, we investigate traces of products and ultraproducts of $\Cstar$-algebras; a topic also of independent interest. Traces of ultraproducts  show up in the recent  work on the Toms-Winter
conjectures: \cite{kirchberg-rordam, toms, etal}.  Given a  $\Cstar$-algebra $A$ let us denote by $\Tb(A)$ the traces  on $A$ of norm at most one (endowed with the weak$^*$-topology).
We prove the following theorem:

\begin{theorem}\label{mainproducts}
Let $A_1,A_2,\dots$ be $\Cstar$-algebras with strict comparison of positive elements by traces. The following are true:
\begin{enumerate}[(i)]
\item
The convex hull of the sets $\Tb(A_1),\Tb(A_2),\dots$ is weak* dense in $\Tb(\prod_{n=1}^\infty A_n)$. 

\item
For any free ultrafilter $U$ in $\N$ we have $\Tb(\prod_U A_i)=\prod_U \Tb(A_i)$.  
\end{enumerate}
Both (i) and (ii) also hold  if we instead assume that the $\Cstar$-algebras $A_1,A_2,\dots$ all have strict comparison of full positive elements by bounded traces and that their primitive spectra are compact.
\end{theorem}

A special case of the theorem above is \cite[Theorem 8]{ozawa}, where the $\Cstar$-algebras are  unital, $\mathcal Z$-stable, and  exact.
Here, $\ZZ$-stability and exactness have been replaced by strict comparison by traces (which we show implies that  ``2-quasitraces are traces").

Here is a brief overview of the paper: 
In Section \ref{commutatorbounds} we introduce the notion of ``commutator bounds" for a $\Cstar$-algebra and discuss its basic properties. We then go over a number of techniques, particularly  a method originally due to Fack, for  proving that a $\Cstar$-algebra  has finite commutator bounds. We take special care to adapt these techniques to the non-unital case.  In Section \ref{strictcomparison} we investigate the property of strict comparison by traces and some variations on it. We show that strict comparison by lower semicontinuous traces implies that (lower semicontinuos) 2-quasitraces are traces. In Section \ref{proofofmain} we prove Theorems \ref{mainpure} and Theorem \ref{mainproducts} stated above. In Section \ref{nilpotents} we use nilpotents of order 2, rather than commutators, to represent trace zero elements of a pure $\Cstar$-algebra. In the last section of the paper we look at multiplicative commutators of unitaries and the kernel of the de la Harpe-Skandalis determinant in a pure $\Cstar$-algebra.


\section{Commutator bounds}\label{commutatorbounds}
Let us start by fixing some notation. Let $A$ be a $\Cstar$-algebra. Let  $A_{\mathrm{sa}}$ and $A_+$ denote the sets of selfadjoint and  positive elements of $A$ respectively. Let  $A^\sim$ denote the unitization of $A$ and  $M(A)$ the   multiplier $\Cstar$-algebra of $A$. 

By a commutator in $A$ we understand an element of the form $[x,y]:=xy-yx$, with $x,y\in A$. We denote the linear span of the commutators by $[A,A]$.
We regard $A/\overline{[A,A]}$ as a Banach space under the quotient norm and  let $\mathrm{Tr}\colon A\to A/\overline{[A,A]}$ denote the quotient map (called the universal trace on $A$). We regard $\Tr$ as also defined on $M_n(A)$ for all $n\in \N$ by  $\mathrm{Tr}((a_{i,j})_{i,j=1}^n)=\Tr(\sum_{i=1}^n a_{ii})$.

We denote by  $\Tb(A)$ the traces  on $A$ of norm at most 1; i.e., the  positive linear functionals on $A$ that vanish on $[A,A]$ and have norm at most 1. It follows from Hahn-Banach's theorem and the Jordan decomposition of bounded traces that  
\[
\|\Tr(a)\|=\sup\{|\tau(a)\mid \tau\in \Tb(A)|\}.
\]
(See \cite[Theorem 2.9]{cuntz-pedersen} and \cite[Proof of Lemma 3.1]{thomsen}).
In particular, 
\[
\overline{[A,A]}=\ker \Tr=\bigcap\{\ker \tau\mid  \tau\in \Tb(A)\} .
\]
We will often write $a\sim_{\Tr}b$ meaning that $\Tr(a-b)=0$; i.e., $a-b\in \overline{[A,A]}$.

In \cite{marcoux06}, Marcoux calls commutator index of a $\Cstar$-algebra the least $m\in \N$ such that
every element $h\sim_{\Tr}0$ is expressible as a sum of $m$ commutators. 
We introduce here a variation on this concept  where only approximation by sums of commutators is required. 
Furthermore, we  keep track of the norms of the  elements appearing in the commutators.

\begin{definition}
Let us say that a $\Cstar$-algebra   $A$ has  commutator bounds $(m,C)$
if for all $h\in \overline{[A,A]}$  and $\epsilon>0$, there exist  $x_1,y_1, \dots,x_m,y_m\in A$ such that
\begin{align}\label{approximation}
\Big\|h-\sum_{i=1}^m[x_i,y_i]\Big\|<\epsilon
\end{align}
and 
\begin{align}\label{control}
\sum_{i=1}^m\|x_i\|\cdot \|y_i\|\leq C\|h\|.
\end{align}
If \eqref{approximation} and \eqref{control} hold with $\epsilon=0$ for some $x_i,y_i\in A$ then we say that $A$ has commutator bounds $(m,C)$ with no  approximations.
\end{definition}

\begin{remark}\label{easierthanozawa}
We can alternatively define commutator bounds without assuming $h\in \overline{[A,A]}$ as follows: for each $h\in A$ and $\epsilon>0$ there exist $x_i,y_i\in A$, with $i=1,\dots,m$ such that
\[
\Big\|h-\sum_{i=1}^m[x_i,y_i]\Big\|\leq \|\Tr(h)\|+\epsilon
\]
and $\sum_{i=1}^m \|x_i\|\cdot \|y_i\|\leq C\|h\|$. 
\end{remark}

Many classes of $\Cstar$-algebras can be shown to have finite commutator bounds:
unital $\Cstar$-algebras with no bounded traces have  finite commutator bounds with no approximations (\cite{pop}); $\Cstar$-algebras of  nuclear  dimension  $m\in\N$ have commutator bounds $(m+1,m+1)$ (\cite[Remark 3.2]{robert-commutators}); 
by Theorem \ref{mainpure} from the introduction (proven below), pure $\Cstar$-algebras whose 2-quasitraces are traces have  commutator bounds $(7,C)$ with no approximations. On the other hand, even among simple unital nuclear $\Cstar$-algebras there are some that  have no finite commutator bounds (\cite[Theorem 1.4]{robert-commutators}).

Before going over  a number of results on the  computation of  commutator bounds, let's discuss an application of this concept to traces of products and ultraproducts. 
Let $A_i$, $i=1, 2, ...$ 
be $\Cstar$-algebras. Recall that the product $\Cstar$-algebra $\prod_{i=1}^\infty A_i$ is the $\Cstar$-algebra of bounded sequences $(a_i)_{i=1}^\infty$, with $a_i\in A_i$ for all $i$. For a given free ultrafilter $U$ of the positive integers, the ultraproduct $\prod_U A_i$ is the quotient of $\prod_{i=1}^\infty A_i$ by the ideal of sequences $(a_i)_{i=1}^\infty$ such that $\lim_U a_i=0$. For each $n\in \N$,
let us view $\tau \in \Tb(A_n)$ as an element of 
$\Tb(\prod_{i=1}^{\infty}  A_i)$
by $\tau((a_i)_i) = \tau(a_n)$.    
For each sequence of traces $(\tau_i )_i$, with  $\tau_i\in \Tb(A_i)$ for all $i$, there is a trace
  on $\prod_U A_i$ given by 
$\Tb(\prod_U A_i)\ni \overline{(a_i)_i} \mapsto  \lim_U \tau_i(a_i)$. Let us denote by $\prod_U \Tb(A_i)$ 
the weak* closure (in $\Tb(\prod_U A_i)$) of the  
set of traces that arise in this way.

\begin{proposition}\label{ultraproducts}
	Let $A_i$, with $i=1,2,\dots$, be $\Cstar$-algebras, all with commutator bounds $(m,C)$  
	for some $m\in \N$ and $C>0$.    Then 
	the convex span of the sets $\Tb(A_i)$, with $i=1,2\dots$, is weakly dense in  $\Tb(\prod _{i=1}^\infty A_i)$. Also, 
	$\prod_{U} \Tb(A_i)=\Tb(\prod_U A_i)$ for any free ultrafilter $U$.
\end{proposition}

\begin{proof}
	The proof of is exactly the same as that of Ozawa's \cite[Theorem 8]{ozawa}, except
	that \cite[Theorem 6]{ozawa} is replaced with this paper's Remark 
	\ref{easierthanozawa}.	(Notice that Ozawa denotes by $\prod_U \Tb(A_i)$ the set of tracial states obtained as limits along the ultrafilter $U$ rather than its weak* closure.)
\end{proof}

Let's now look into permanence properties for the commutator bounds:

\begin{proposition}
Let $A$ be a $\Cstar$-algebra and $I$ a closed two-sided ideal of $A$.
\begin{enumerate}[(i)]
\item
If $A$ has commutator bounds $(m,C)$  then so do  $I$ and $A/I$.

\item
If $I$ and $A/I$ have commutator bounds  $(m,C)$ and $(n,D)$, respectively,    then $A$ has commutator bounds
$(m+n,C+D)$.

\item
Let $A$ be the inductive limit of $\Cstar$-algebras $(A_\lambda)_{\lambda\in \Lambda}$, each with commutator bounds $(m,C)$. Then 
 $A$ has commutator bounds $(m,C)$ too.
\end{enumerate}
\end{proposition}

\begin{proof}
(i) Let $h\in \overline{[I,I]}$. Since $A$ has commutator
bounds  $(m, C)$, we can find $x_1,y_1,\dots,x_m,y_m\in A$  that satisfy \eqref{approximation} and \eqref{control}.
Let $(e_\lambda)_\lambda$ be an approximately central approximate unit of $I$. Then for $\lambda$ large enough $x_i'=x_ie_\lambda$ and
$y_i'=y_ie_\lambda$  satisfy \eqref{approximation} and  \eqref{control} and belong to $I$.

Let us suppose now that $h \in A/I$ and $h \sim_{\Tr}0$. It suffices to assume that $\|h\|=1$.
Let $\epsilon > 0$ be given.
By \cite[Lemma 2.1 (i)]{kerdet}, there exists a  lift $h'\in A$  of $h$ such that 
$h' \sim_{\Tr} 0$ and $\| h' \| \leq  1 + \frac \epsilon 2$.
Since $A$ has commutator bounds $(m, C)$,  there exist
  $x_1 , y_1,\dots,x_m,y_m  \in A$  such that 
\begin{gather}\label{hprime}
\Big\| \Big(1+\frac \epsilon 2\Big)^{-1}h' - \sum_{i=1}^m [x_i, y_i]\Big \|  < \frac \epsilon 2,\\
\sum_{i=1}^m \| x_i \| \cdot \| y_i \|  \leq C.\nonumber
\end{gather}
Observe now that $\|h'-(1+\frac \epsilon 2)^{-1}h'\|=\|h'\|(1-(1+\frac \epsilon 2)^{-1}\|\leq \frac \epsilon 2$. Hence, from \eqref{hprime} we get 
\[
\Big\|h'-\sum_{i=1}^m[x_i,y_i]\Big\|<\epsilon.
\]
Thus, the images in the quotient $A/I$ of $x_1,y_1,\dots x_m,y_m$, and $h'$,  satisfy \eqref{approximation} and  \eqref{control}, as desired.

(ii) Let $h \in A$ with $h \sim_{\Tr} 0$.
 Let $h'\in A/I$ denote the image of $h$ in $A/I$.
 Let $\epsilon > 0$ be given. 
 Since the quotient $\Cstar$-algebra $A/I$ has commutator bounds $(n, D)$, 
 there exist $x'_1, y'_1,\dots,x_n',y_n' \in A/I$ such that 
$\| h' - \sum_{i=1}^n [x'_i, y'_i] \| < \frac \epsilon {3}$ 
and $\sum_{i=1}^n \| x'_i \|\cdot \| y'_i \| \leq D \| h' \|$.
For  $i=1,\dots,n$, let us find lifts $x''_i$ and  $y''_i$   in $A$ of  $x'_i$ and $y'_i$ , respectively, such that 
 $\| x''_i \| = \| x'_i \|$ and $\| y''_i \| = \| y'_i \|$. 
 Let $(e_{\lambda})_{\lambda}$ be an approximately central approximate
 unit of  the ideal $I$. We can choose an index $\lambda$ such that 
 \[
 \| h - e_{\lambda} h e_{\lambda} - (1 - e_{\lambda}) h (1 - e_{\lambda}) 
 \|  < \frac \epsilon {3},
 \]
and 
\[
 \Big\| (1 - e_{\lambda} ) h (1 - e_{\lambda}) - 
 \sum_{i=1}^n [ x_i, y_i]\Big \| < \frac \epsilon {3},
\] 
where
 $x_i = (1 - e_{\lambda} )^{1/2} x''_i (1 - e_{\lambda} )^{1/2}$
 and 
 $y_i = (1 - e_{\lambda} )^{1/2} y''_i (1 - e_{\lambda} )^{1/2}$ 
 for $i=1,\dots,n$. Note that 
 \[
 \sum_{i=1}^n \| x_i \|\cdot \| y_i \| \leq 
\sum_{i=1}^n \| x''_i \| \cdot \| y''_i \| \leq D \| h' \| \leq D \| h \|.
\]
 Since $e_{\lambda} h e_{\lambda} \in I$ and the ideal $I$ has commutator
 index $(m, C)$, 
 we can find elements $x_{n+1}, y_{n+1},\dots,x_{n+m},y_{n+m} \in I$  such that 
\begin{gather*}
\Big\| e_{\lambda} h e_{\lambda} - \sum_{i=n+1}^{n+m} [x_i, y_i] 
 \Big\| < \frac\epsilon {3},\\
 \sum_{i=n+1}^{n+m} \| x_i \|\cdot \| y_i \| \leq C \| e_{\lambda} h 
 e_{\lambda} \| \leq C \| h \|.
 \end{gather*}
 Hence, $\sum_{i=1}^{m+n} \| x_i \| \cdot \| y_ i \| \leq (C+D)\| h \|$ and 
\begin{multline*}
\Big\| h - \sum_{i=1}^{m+n} [x_i, y_i] \Big\| 
\leq \|h-e_{\lambda} h e_{\lambda} -(1 - e_{\lambda}) h (1 - e_{\lambda})\|\\
+\Big\|(1 - e_{\lambda}) h (1 - e_{\lambda})-\sum_{i=1}^n [ x_i, y_i]\Big\|+
\Big\|e_{\lambda} h e_{\lambda}-\sum_{i=n+1}^{n+m} [x_i, y_i] \Big \|<\epsilon.
\end{multline*}

(iii) Since we have already shown that the commutator bounds pass to quotients, we may assume that
$A_\lambda\subseteq A$ for all $\lambda$ and that $\overline{\bigcup_{\lambda\in \Lambda} A_\lambda}=A$.

Let $h \in A$ be such that $h \sim_{\Tr} 0$.  Let $\epsilon > 0$ be given. Let us prove the existence of $x_1,y_1,\dots,x_m,y_m\in A$ 
satisfying \eqref{approximation} and \eqref{control}.
It is clear that we may  reduce ourselves to the case   $\|h\|=1$. 
We claim there exists   $\lambda\in \Lambda$ and a contraction $h'\in A_\lambda$ such that $\|h-h'\|<\epsilon/2$ and $h'\sim_{\Tr}0$ in $A_\lambda$. 
To prove this, we first  approximate $h$ sufficiently by a finite sum of commutators:
$\| h - \sum_{j=1}^k [v_j, w_j] \| < \epsilon/4$. Next, we
choose $\lambda \in \Lambda$ and $v'_1, w'_1,\dots,v_k',w_k' \in A_\lambda$  such that  $\|v_j-v_j'\|$ and $\|w_j-w_j'\|$ are sufficiently small for all $j$, so that
$\|h-\sum_{j=1}^k [v'_j, w'_j ]\|<\epsilon/4$. Finally, we set 
\[
h'=\Big(1+\frac \epsilon 4\Big)^{-1} \sum_{j=1}^k [v'_j,w'_j ].
\] 
Notice then
that 
\[
\|h-h'\|\leq \Big\|h-\Big(1+\frac \epsilon 4\Big) h'\Big\|+\frac \epsilon 4\|h'\|<\frac \epsilon 2,
\] 
and that $h'\sim_{\Tr}0$ in $A_\lambda$, as desired.  

Since $A_\lambda$ has commutator bounds $(m, C)$, there exist  
$x_1, y_1,\dots,x_m,y_m \in A_\lambda$
such that 
$\| h' - \sum_{j=1}^m [x_i, y_i ] \| 
< \epsilon/2$ and 
$\sum_{i=1}^m \| x_i \|\cdot \| y_i \| \leq C \| h' \|\leq C$. These are the desired elements. 
 \end{proof}

It is possible to reduce  the number of commutators by passing from a $\Cstar$-algebra $A$ with commutator bounds $(m,C)$ to a matrix algebra $M_n(A)$.  This, however, is  achieved at the expense of increasing the constant $C$.	   
Marcoux obtains such a reduction  for unital $\Cstar$-algebras in   \cite[Lemma 4.1]{marcoux09}  and
\cite[Lemma 2.2]{marcoux06}. Here,  we cover the non-unital case and  give explicit bounds  for the norms of the commutators.

\begin{lemma}\label{diagonalcommutator}
Let $d_1,\dots,d_n$ in $A$ be such that $\sum_{i=1}^n d_i=0$. Then there exist
$X,Y\in M_n(A)$, with $\| X \| \cdot \|Y\|\leq 4n \max_i \| d_i \|$, such that  the main diagonal of $[X,Y]$ equals $(d_1,\dots,d_n)$.
\end{lemma}
\begin{proof}
	For  $k=1,\dots,n-1$, let $s_k= \sum_{i=1}^k d_{i}$. Let us write $s_k=q_kr_k$, for some $q_k,r_k\in A$ such that
	$\|q_k\|= \|r_k\|=\|s_k\|^{\frac 1 2}$ (e.g.,  $q_k=v|s_k|^{\frac 1 2}$ and $r_k=|s_k|^{\frac 1 2}$, where   $s_k = v |s_k|$ is the polar decomposition  of 
	$s_k$ in $A^{**}$).  
	Let
	\[
	X = 
	\begin{pmatrix}
	0 & q_1 &&&\\
	&r_1&q_2&&\\
	&&r_2&\ddots&\\
	&&&\ddots& q_{n-1}\\
	&&&& r_{n-1}
	\end{pmatrix},
	\quad	
	Y = 
	\begin{pmatrix}
	0 &&&&\\
	r_1&q_1&&&\\
	&r_2&q_2&&\\
	&&\ddots& \ddots &\\
	&&& r_{n-1}&q_{n-1}
	\end{pmatrix}.
	\]
	A straightforward computation shows that $X$ and $Y$ are as required.
\end{proof}
	
\begin{lemma}\label{zerodiagonal}
Let $h\in M_n(A)$, with  $h=(h_{j,k})_{j,k}$.  
 Suppose that $h_{j,j}=[x_j,y_j]$ for  some $x_j,y_j\in A$, for  $j=1,\dots,n$. Then $h=[X,Y]$ for some $X,Y\in M_n(A)$    such that  
 $\| X \|\cdot \|Y\| \leq 36n^2 (n-1)  \| h \| + 3n\sum_{k=1}^n 
		\| x_k \|\cdot  \| y_k \|$. 
\end{lemma}
\begin{proof}
Let $h\in M_n(A)$ and $x_j,y_j\in A$, with $j=1, \dots,n$, be as in the statement of the lemma. Let us also assume that $x_j$ is a contraction for all $j$  (replacing $y_j$ with $\| x_j \| y_j$ if necessary). Let $\lambda_j = 3(j-1)$ 
and $d_j  = x_j + \lambda_j 1 \in A^{\sim}$,
	for $j=1,\dots,n$.
	Notice that the spectrum of $d_j$ is contained in $\{z\in \C\mid |z-\lambda_j|<1\}$ for all $j$.
	In particular, the spectra of the $d_j$'s are pairwise disjoint.
	Notice also  that $[d_j, b] = [x_j, b]$ for all $b \in A^{\sim}$  and
	  $j=1, \dots,n$. 
	  Let us fix $k,j$ with $1\leq k,j\leq n$ and $k\neq j$.
	By \cite[Corollary 3.2]{herrero},  
	the Rosenblum operator $T_{k,j} \colon A^{\sim} \rightarrow A^{\sim} $, 
	defined by $ T_{k,j}(b)=d_k b - b d_j$,  is invertible.
	Let
	$b_{k,j} \in A^{\sim}$ be such that 
	$T_{k,j}(b_{k,j}) = d_k b_{k,j} - b_{k,j} d_j = h_{k,j}$.
	Since $h_{k,j} \in A$, and by our choice of $\lambda_k$ and  $\lambda_j$,
	we must have that $b_{k,j} \in A$. 	
	
	Let us define $b_{k,k} := y_k$ for $k=1,\dots,n$.  
Let $X, Y \in M_n(A)$ be given by 
\begin{align*}
	X = \mathrm{diag}(d_1,\dots,d_n),\quad
	Y = (b_{k,j})_{k,j=1}^n. 
\end{align*}
It is a straightforward computation show that $[X,Y]=a$ (see \cite[Lemma 2.2]{marcoux06}).

Let us find bounds for the norms of $X$ and $Y$. The bound $\| X \|=\max_j \|d_j\|\leq 3 n$ is straightforward. In order to bound $\|Y\|$, we first estimate $\|b_{k,j}\|$. Fix $k,j$ such that $1\leq k,j\leq n$ and $k\neq j$. By \cite[Corollary 3.20]{herrero}, 
\begin{equation}\label{equ:RosenblumInverseEquation}
b_{k,j} = \frac{1}{2 \pi i} \int_{\Gamma_k} (d_k - \alpha 1)^{-1} h_{k,j} (d_j - \alpha 1)^{-1} d \alpha, 
\end{equation} 
where $\Gamma_k$ is the positively oriented simple closed contour given 
by $\Gamma_k(t) = \lambda_k + \frac{3}{2} e^{it} $, for $t \in [0, 2 \pi]$.
We have
\[
(d_k - \alpha 1)^{-1} = \frac{1}{\lambda_k - \alpha} 
\left( 1 - \frac{d_k - \lambda_k 1}{ \alpha - \lambda_k} \right)^{-1}.
\]
But $\| d_k - \lambda_k 1 \| = \| x_k \| \leq 1$  and
$| \alpha - \lambda_k | = 3/2$ for all $\alpha \in \Gamma_k$.
So, $\| \frac{d_k - \lambda_k 1}{ \lambda_k - \alpha} \| \leq 1 / (3 /2) = 2/3$, from which we deduce that
	\[
	\| (d_k - \alpha 1)^{-1} \| 
	\leq \frac{2}{3} \sum_{l=0}^{\infty} (2/3)^l 
	= 2
	\] 
	for all $\alpha \in \Gamma_k$.
	Next, since $k \neq j$, 
	$|\lambda_j - \alpha | \geq 3/2$ for all $\alpha \in \Gamma_k$.  Also, 
	$\| d_j - \lambda_j 1 \| = \| x_j \| \leq 1$.
Hence, 
\[
\|(d_j - \alpha 1)^{-1}\| = \frac{1}{\|\lambda_j - \alpha\|} 
\left\|\Big( 1 - \frac{d_j - \lambda_j 1}{ \alpha - \lambda_j} \Big)^{-1}\right \|
\leq \frac{2}{3} \sum_{l=0}^{\infty} (2/3)^l
= 2,
\]	
	for all $\alpha \in \Gamma_k$. 
	Thus, $\| (d_k - \alpha 1)^{-1} \|\leq 2$ and $\| (d_j - \alpha 1)^{-1} \| \leq 2$ for all $\alpha \in \Gamma_k$.
	From this and (\ref{equ:RosenblumInverseEquation}), we get 
	\[\| b_{k,j} \| \leq
	(\frac{1}{2 \pi}) 4 \| h_{k,j} \| \cdot \mathrm{length}(\Gamma_k)
	=  
	(\frac{1}{2 \pi}) 4 \| h_{k,j} \| (6 \pi  )
	= 12 \| h_{k,j} \| 
	\leq 12 \| h \|,\]
	for all
	$k \neq j$.
	Recall that, by our conventions,
	$\| b_{k,k} \| = \| y_k \| = \| x_k \|\cdot  \| y_k \|$.
	Then 
	$
	\| Y \| \leq n (n - 1) 12 \| h \|  + \sum_{k=1}^n \| x_k \|\cdot  \| y_k \|
	$. This, together with $\|X\|\leq 3n$, proves the lemma.
\end{proof}

\begin{theorem}\label{matrixreduction}
Let $A$ be a $\Cstar$-algebra and $n\in \N$.
\begin{enumerate}[(i)]
\item
		If $M_n(A)$ has commutator bounds $(m,C)$ (with no approximations) then $A$ has commutator bounds $(mn^2,Cn)$ (with no approximations).
\item
		If $A$ has commutator bounds $(m,C)$ (with no approximations), then $M_n(A)$ has commutator bounds $(2,C')$ (with no approximations) for all $n\geq m$, where 
		$C'\leq 36 n^3 + (2C - 36)n^2 + n$.
\end{enumerate}
\end{theorem}

\begin{proof}
	(i) Let $h \in A$ with $h \sim_{\Tr} 0$.
	Then $a \otimes 1_n \in M_n(A)$ and $a \otimes 1_n 
	\sim_{\Tr} 0$ in $M_n(A)$. 
	Let $\epsilon > 0$ be given.  Since $M_n(A)$ has commutator bounds 
	$(m, C)$, there exist $X_j, Y_j \in M_n(A)$ ($1 \leq j \leq m$) such that
	\begin{gather}\label{matrixtoA}
		\| h \otimes 1_n - \sum_{j=1}^m [X_j, Y_j] \| < \epsilon,\\
		\sum_{j=1}^m \| X_j \| \cdot \| Y_j \| \leq C \| a \otimes 1_n \|  = C \| h \|.
		\label{matrixbound}
	\end{gather}
	Averaging  along the main diagonal  in \eqref{matrixtoA} we get 
	\[
	\| h - \frac{1}{n}\sum_{j=1}^m \sum_{k,l=1}^n [x_{j,k,l}, y_{j,l,k}]\| < \epsilon,
	\] 
	where $X_j = (x_{j,k,l})_{k,l=1}^n$
	and $Y_j = (y_{j,k,l})_{k,l=1}^n$. On the other hand, using \ref{matrixbound} we get
	\[
	\frac{1}{n}\sum_{j=1}^m \sum_{k,l=1}^n \|x_{j,k,l}\|\cdot \|y_{j,l,k}\|\leq 
	\frac{1}{n}\sum_{j=1}^m \sum_{k,l=1}^n \|X_j\|\cdot \|Y_j\|=n\sum_{j=1}^m\|X_j\|\cdot \|Y_j\|\leq nC\|h\|,
	\]
	as required. The same arguments above, but with $\epsilon=0$, prove the result for commutator bounds with no approximations.
	
	(ii) Let us deal with the case of commutator bounds with no approximations. 
Let $h\in M_n(A)$ be such that $h\sim_{\Tr}0$.  Then 
	$\sum_{i=1}^n h_{i,i}\sim_{\Tr}0$ in $A$. But $A$ has commutator bounds $(m,C)$ with no approximations. Hence, $\sum_{i=1}^n h_{i,i}=\sum_{i=1}^m [x_i,y_i]$ for some $x_i,y_i\in A$.    
	By Lemma \ref{diagonalcommutator}, there exist  $X_1,Y_1 \in M_n(A)$ such that   the entries along the main diagonal of $[X_1,Y_1]$ equal 
	\[
	 h_{1,1}-[x_1,y_1],\dots,h_{m,m}-[x_m,y_m],h_{m+1,m+1},\dots, h_{n,n}.
	\] Now, by Lemma \ref{zerodiagonal}, $h-[X_1,Y_1]=[X_2,Y_2]$ for some $X_2,Y_2\in M_n(A)$. The bound on $C'$  follows from the norm bounds in Lemmas \ref{diagonalcommutator} and \ref{zerodiagonal}.

In the case that the algebra $A$ has commutator bounds $(m,C)$ with approximations,   the initial element $h\sim_{\Tr}0$ can be slightly perturbed along the main diagonal so that, for the perturbed element,  the sum of the  diagonal entries  is exactly a sum of $m$ commutators.  The   arguments above then show that  the perturbed element is a sum of two commutators.
\end{proof}

The proof of Theorem \ref{fackstechnique} below relies on a technique first used by Fack in \cite{fack}. Despite its technical statement, Theorem \ref{fackstechnique}
constitutes our main tool in proving that a $\Cstar$-algebra has finite commutator bounds with no approximations. Before stating the theorem, we introduce some definitions and prove a lemma.

Let us define the direct sum  of positive elements in $A\otimes \mathcal K$.   Fix isometries $v_1,v_2\in B(\ell_2)$ generating the Cuntz algebra $\mathcal O_2$. Let us regard them  as elements of the multiplier algebra  $M(A\otimes \mathcal K)$  via the natural embeddings $1\otimes B(\ell_2)\subseteq M(A)\otimes B(\ell_2)\subseteq M(A\otimes \mathcal K)$. Then, given $a,b\in (A\otimes \mathcal K)_+$, let us  define $a\oplus b\in (A\otimes \mathcal K)_+$ by $a\oplus b=v_1av_1^*+v_2bv_2^*$.

Next, let us introduce a preorder relation on the positive element of a $\Cstar$-algebra. Let $a,b\in A_+$.
Let us write $a\precsim b$ if $a=x^*x$ and $xx^*\in \her(b)$ for some $x\in A$. This relation is called Blackadar's relation in \cite{ortega-thiel-rordam}. It can be alternatively described as saying that the right ideal  $\overline{aA}$ embeds into $\overline{bA}$ as a Hilbert module. (Thus,  it   is clearly transitive.) We will make repeated use of the following fact  (see \cite[Proposition 4.6]{ortega-thiel-rordam}): Say $a=x^*x$ and $xx^*\in \her(b)$. Let $x=v|x|$ be the polar decomposition of $x$ in $A^{**}$. Then $vy\in \her(b)$ for any $y\in \her (a)$.

\begin{lemma}\label{aplusnb}
Let $a,b\in A_+$ be such that $a\precsim b^{\oplus n}$ (in $A\otimes \mathcal K$). Then for all $h\in \her(a)$ there exist $z_1,w_1,\dots,z_n,w_n\in A$ and  $h'\in \her(b)$ such that $h=\sum_{i=1}^n [z_j,w_j]+h'$,   $\|z_j\|\cdot \|w_j\|\leq \|h\|$ for all $j$,  and $\|h'\|\leq n\|h\|$. 
\end{lemma}
\begin{proof}
Let us  regard $A$  as  a subalgebra of $M_n(A)$ embedded in the top left corner. The assumption $a\precsim b^{\oplus n}$ can be rephrased as $a\precsim b\otimes 1_n$ in $M_n(A)$. That is, there exists  $x\in M_n(A)$  such that
$a=xx^*$ and $xx^*\in \her(b\otimes 1_n)$. Let $x=v|x|$, with $v\in M_n(A)''$,  be the polar decomposition of $x$. Let $(v_1,\dots,v_n)$ denote the first row of $v$ (the rest of the rows are 0).
Finally, let us write $h=h_1h_2$, with $h_1,h_2\in \her(a)$ such that $\|h_1\|\cdot \|h_2\|=\|h\|$ (e.g., as in the proof of Lemma \ref{diagonalcommutator}). Then
\[
h=h_1\Big(\sum_{j=1}^n v_j^*v_j\Big)h_2=\sum_{j=1}^n [h_1v_j^*,v_jh_2]+\sum_{j=1}^n v_jh_2h_1v_j^*.
\]
Hence, the elements  $z_j=h_1v_j^*$,  $w_j=v_jh_2$, for $j=1,\dots, n$, and $h'=\sum_{j=1}^n v_jh_2h_1v_j^*$ are as required.
\end{proof}

\begin{theorem}\label{fackstechnique}
	Let $A$ be a $\Cstar$-algebra and $e_0\in A_+$ a strictly positive element. Suppose that the following are true:
	\begin{enumerate}[(i)]
		\item
		There exist an integer $L \geq 1$  and pairwise  orthogonal positive elements $e_1,e_2,\dots \in A_+$   such that  
		$e_j \preceq \bigoplus^L e_{j+1}$  for   $j=0,1,\dots$.
		
		\item
		There exist constants $C>0$, $M\in \N$, and $0<\lambda<1$, such that 
		for     all $j\in \{0,1,\dots\}$ and $h\in \her(e_j)$ such that $h\sim_{\Tr}0$ (in $\her(e_j)$), there exist $x_1,y_1,\dots,x_M,y_M\in \her(e_j)$  such that
		\[
		\|h-\sum_{i=1}^M[x_i,y_i]\|\leq \lambda \|h\|,
		\] 
		and $\|x_i\|\cdot \|y_i\|\leq C \| h \|$ for all $i$.
	\end{enumerate}
	Then $A$ has finite  commutator bounds $(\overline{M},\overline{C})$ with no approximations, where $\overline M$ and $\overline C$ depend  only 
	on $L,M,\lambda,C$.
\end{theorem}

\begin{proof}
	Let us choose $L_1\in \N$ such that 	$\lambda^{L_1} < 1/(2L)$.
	From hypothesis (ii) we deduce the following: 
	\begin{enumerate}
		\item[(ii')]
	for all  $j\in \{0,1,\dots\}$ and  $h \in \her(e_j)$ such that $h \sim_{\Tr} 0$ (in $\her(e_j)$), there exist
	$x_1, y_1,\dots,x_{L_1M},y_{L_1M} \in \her(e_j)$  such that
	\begin{equation*}
	\| h - \sum_{i=1}^{L_1 M} [x_i, y_i] \| < \lambda^{L_1} \| h \| < \frac{1}{2 L} \| h \|,
	\end{equation*}
	and $\| x_i \|, \| y_i \| \leq C^{\frac 1 2} \| h \|^{\frac 1 2}$ for all $i$.  
	\end{enumerate}
		
Let $h  \in A$ be such that $h\sim_{\Tr}0$. By hypothesis (i) and Lemma \ref{aplusnb}, we have 
\[
h = \sum_{j=1}^{L} [z_j,w_j] + h_1,
\]
where $h_1\in \her(e_1)$ and $\|h_1\|\leq L\|h\|$. By (ii') above, applied in the hereditary algebra $\her(e_1)$, there exist $x_1^{(1)},y_1^{(1)},\dots,x_{L_1M}^{(1)},y_{L_1M}^{(1)}\in \her(e_1)$  and $h_1'\in \her(e_1)$, such that
\[
h_1=\sum_{i=1}^{L_1M} [x_i^{(1)},y_i^{(1)}]+h_1',
\]
 and $\|h_1'\|\leq \frac{1}{2L}\|h_1\|\leq \frac{\|h\|}{2}$. Again by hypothesis (i) and Lemma \ref{aplusnb}, 
we have 
\[
h_1'=\sum_{j=1}^{L} [z_j^{(1)},w_j^{(1)}] + h_2,\]
where $h_2\in \her(e_2)$, $z_j^{(1)},w_j^{(1)}\in \her(e_1+e_2)$ for all $j$,  and $\|h_2\|\leq L\|h_1'\|$. Applying (ii') in $\her(e_2)$, we get
$h_2=\sum_{i=1}^{L_1M} [x_i^{(2)},y_i^{(2)}] + h_2'$,
with  $\|h_2'\|\leq \frac{\|h_1'\|}{2}\leq \frac{\|h\|}{4}$. Continuing in this way, we construct, for each $n\in \N$, elements  
\begin{enumerate}
	\item
	$h_n,h_n'\in \her(e_n)$ such that $\|h_n'\|\leq  \frac 1 {2^n}\|h\|$, $\|h_n\|\leq \frac{L}{2^{n-1} }\|h\|$,
	\item
$x_1^{(n)},y_1^{(n)},\dots,x_{L_1M}^{(n)},y_{L_1M}^{(n)}\in \her(e_n)$,  such that 
$\|x_i^{(n)}\|,\|y_{i}^{(n)}\|\leq C^\frac 1 2\|h_n\|^\frac 1 2$ for all $i$, and 
\[
h_n=\sum_{i=1}^{L_1M} [x_i^{(n)},y_i^{(n)}]+h_n',
\] 
\item 
$z_1^{(n)},w_1^{(n)},\dots,z_L^{(n)},w_L^{(n)}\in \her(e_n+e_{n+1})$, such that 
$\|z_j^{(n)}\|,\|w_j^{(n)}\|\leq  \|h_n'\|^{\frac 1 2}$ for all $j$, and 
\[
h_n'=\sum_{j=1}^{L} [z_j^{(n)},w_j^{(n)}] + h_{n+1}.
\]
\end{enumerate} 
It follows that 
\[
h_1=
\sum_{k=1}^n\sum_{i=1}^{L_1M}[x_i^{(k)},y_i^{(k)}]+
\sum_{k=1}^n\sum_{j=1}^{L}[z_j^{(k)},w_j^{(k)}]+h_{n+1}.
\]
We can gather terms belonging to orthogonal hereditary subalgebras and define
\begin{align*}
X_i =\sum_{n=1}^\infty x_i^{(n)},\quad
Y_i =\sum_{n=1}^\infty y_i^{(n)},
\end{align*}
for $i=1,\dots,L_1M$, and
\begin{align*}
&& Z_{j,0} &=\sum_{n\hbox{ \small{odd}}}^\infty z_j^{(n)},  &   Z_{j,1} &=\sum_{n\hbox{ \small{even}}}^\infty z_j^{(n)},&\\
&& W_{j,0}  &=\sum_{n\hbox{ \small{odd}}}^\infty w_j^{(n)}, &  W_{j,1} &=\sum_{n\hbox{ \small{even}}}^\infty w_j^{(n)},&
\end{align*}
for $j=1,\dots,L$.  Note that the terms in the series defining the elements $X_i,Y_,Z_{j,k},W_{j,k}$ are pairwise orthogonal. Also, the norm estimates on the elements $x_i^{(n)},y_i^{(n)},z_j^{(n)},w_j^{(n)}$
guarantee that these series converge. Furthermore, it is clear that norm estimates on   $X_i,Y_,Z_{j,k},W_{j,k}$ can be obtained from those on $x_i^{(n)},y_i^{(n)},z_j^{(n)},w_j^{(n)}$.  We have
\[
h=\sum_{j=1}^{L} [z_j,w_j]+\sum_{i=1}^{L_1M} [X_i,Y_i]+\sum_{j=1}^L[Z_{j,0},W_{j,0}]+\sum_{j=1}^L[Z_{j,1},W_{j,1}].
\]
This shows that $A$ has commutator bounds $(3L+L_1M,C')$ with no approximations, for some  $C'$.
\end{proof}


\section{Strict comparison of positive elements}\label{strictcomparison}
Here we review and explore the strict comparison of positive elements by traces and 2-quasitraces. Some of these results will be used in the proof of Theorem \ref{mainpure} in the next section.


Let us start by recalling the definition of the Cuntz semigroup.
Let $A$ be a $\Cstar$-algebra. Let  $a,b\in (A\otimes \mathcal K)_+$. Let us write $a\precsim_{\Cu}b$ if 
$d_n^*bd_n\to a$ for some $d_n\in A\otimes \mathcal K$. In this case we say that $a$ is Cuntz smaller than $b$.
Let us write $a\sim_{\Cu}b$ if $a\precsim_{\Cu}b$ and $b\precsim_{\Cu}a$, in which case we say that $a$ and $b$ are Cuntz equivalent.
Let $[a]$ denote the Cuntz class of $a\in (A\otimes \mathcal K)_+$. 

The Cuntz semigroup of $A$, denoted by $\Cu(A)$, is defined as the 
quotient set $(A\otimes \mathcal K)/\!\!\sim_{\Cu}$, endowed the following order and addition:
$[a]\leq [b]$ if $a\precsim_{\Cu}b$ and  $[a]+[b]=[a\oplus b]$, with the direct sum  $a\oplus b\in (A\otimes \mathcal K)_+$ as defined in the previous section.  The reader is referred to \cite{ara-perera-toms} and \cite{antoine-perera-thiel} for the basic theory of the Cuntz semigroup (some of which will be used below).

Let us denote by $\Text(A)$  the cone of lower semicontinuous $[0,\infty]$-valued traces on $A$; i.e., 
the lower semicontinuous maps $\tau\colon A_+\to [0,\infty]$ that are additive, homogeneous, map 0 to 0, and satisfy that $\tau(x^*x)=\tau(xx^*)$ for all $x\in A$. 
Let us denote by $\QT(A)$ the lower semicontinuous $[0,\infty]$-valued 2-quasitraces on $A_+$.  
Traces and 2-quasitraces extend uniquely to traces and quasitraces on $(A\otimes \mathcal K)_+$,
and we shall regard them as defined on this domain (see \cite[Remark 2.27 (viii)]{blanchard-kirchberg}).  
Recall (from the previous section) that we denote by $\Tb(A)$ the convex set of traces on $A$ of norm at most 1.

A topology on $\QT(A)$ can be defined as follows:
Let $(\tau_\lambda)_\lambda$  be a net  in $\QT(A)$ and $\tau\in \QT(A)$. Let us say that $\tau_{\lambda}\to \tau$
if for any $a\in (A\otimes \mathcal K)_+$ and $\epsilon>0$   we have  
\begin{align}\label{topology}
\limsup_\lambda \tau_\lambda((a-\epsilon)_+)\leq \tau(a)\leq \liminf_{\lambda} \tau_{\lambda}(a).
\end{align}
In this way $\QT(A)$ is a compact Hausdorff space and $\Text(A)$ and $\Tb(A)$ are  closed subsets  of $\QT(A)$
(see \cite[Subsections 3.2 and 4.1]{ers}) .

The dimension function associated to $\tau\in \QT(A)$ is defined as 
\[
d_\tau(a)=\lim_n \tau(a^{1/n})=\|\tau\mid_{\her(a)}\|,
\] 
for all $a\in (A\otimes \mathcal K)_+$.
The value $d_\tau(a)$ depends only on the Cuntz class of the positive element $a$. Thus, by a slight abuse of notation, we also  write $d_\tau([a])$.

The ordered semigroup 
$\Cu(A)$ is called \emph{almost unperforated} if 
\[
(k+1)x\leq ky\Rightarrow x\leq y
\]
 for all $k\in \N$  and    $x,y\in \Cu(A)$. 
The ordered semigroup $\Cu(A)$ is called \emph{almost divisible} if
for all $n\in \N$, $x\in \Cu(A)$ and $x'\ll x$ (i.e., $x'$ compactly contained in $x$), there exists $y\in \Cu(A)$ such that 
\[
ny\leq x\hbox{ and }x'\leq (n+1)y.
\] 
The $\Cstar$-algebra $A$ is called \emph{pure} if $\Cu(A)$ is both almost unperforated and almost divisible. By \cite{rordamZ}, $\Cstar$-algebras that absorb tensorially the Jiang-Su algebra are pure. There are, however, tensorially prime pure $\Cstar$-algebras.

It is shown in \cite[Proposition 6.2]{ers} (and in  \cite[Corollary 4.6]{rordamZ} for simple $\Cstar$-algebras) that   almost unperforation in $\Cu(A)$ is equivalent to the property of strict comparison of positive elements by 2-quasitraces. We consider here the following generalization of the latter property:

\begin{definition}\label{strictbyK}
Let $A$ be a $\Cstar$-algebra and   $K\subseteq \QT(A)$   a compact subset. Let us say that $A$ has strict comparison of positive elements by 2-quasitraces in $K$
if   $d_\tau(a)\leq (1-\gamma)d_\tau(b)$  for all $\tau\in K$ and some $\gamma>0$ implies that $[a]\leq [b]$ for all $a,b\in (A\otimes \mathcal K)_+$. 
 \end{definition}

For $K=\QT(A)$, this notion agrees with the strict comparison of positive elements mentioned above. Another  case of interest is $K=\Text(A)$. In this case we say that $A$ has strict comparison of positive elements by traces. If $A$ is unital or more generally   $\mathrm{Prim}(A)$  
is compact, it is also interesting to consider the property of strict comparison by traces  restricted to full positive elements only  (i.e., those generating $A$ as a two-sided ideal). 
Let us define this more formally:

\begin{definition}
Let $A$ be a $\Cstar$-algebra such that $\mathrm{Prim}(A)$ is  compact. Let us say that $A$ has strict comparison of full positive elements by traces if $d_\tau(a)\leq (1-\gamma)d_\tau(b)$ for all $\tau\in \Text(A)$ and some $\gamma>0$ implies that $[a]\leq [b]$ for all 
  full positive elements $a,b\in (A\otimes \mathcal K)_+$.
\end{definition}

\begin{lemma}\label{epsilondelta}
Let $K\subseteq \QT(A)$ be compact and $a,b\in (A\otimes \mathcal K)_+$.
 Suppose that $d_\tau(a)\leq (1-\gamma)d_\tau(b)$ for all $\tau\in K$ and some $\gamma>0$.  Then for each $\epsilon>0$
there exists $\delta>0$ such that 
\[
d_\tau((a-\epsilon)_+)\leq \Big(1-\frac \gamma 2\Big)d_\tau((b-\delta)_+),
\]
for all $\tau\in K$.
\end{lemma}
\begin{proof}
As shown in the  proof of \cite[Lemma 5.11]{ers},  we have
\begin{align*}
\overline{\{\tau\in K\mid d_{\tau}((a-\epsilon)_+)>1\}} \subseteq 
\{\tau\in K\mid \Big(1-\frac \gamma 2\Big)d_{\tau}(a)>1\}.
\end{align*}
The left side is compact while the right side is covered by the open sets  
\[
\{(\tau\in K\mid \Big(1-\frac \gamma 2\Big)d_\tau\Big((b-\frac 1 n)_+\Big)>1\}, \quad n=1,2,\dots.
\]
Thus, one of these open sets covers $\{\tau\in K\mid d_{\tau}((a-\epsilon)_+)>1\}$. This proves the lemma. 
\end{proof}

\begin{lemma}\label{equalontraces}
Let $A$ be a $\Cstar$-algebra with strict comparison of positive elements by $K$, where   $K\subseteq \QT(A)$ is compact. Let $a,b\in (A\otimes \mathcal K)_+$. 
\begin{enumerate}[(i)]
\item
If
$d_\tau(a)\leq d_\tau(b)$ for all $\tau\in K$ then $d_\tau(a)\leq d_\tau(b)$ for all $\tau\in \QT(A)$.
\item
If  $\tau(a)\leq \tau(b)$ for all $\tau\in K$ then $\tau(a)\leq \tau(b)$ for all $\tau\in \QT(A)$.
\end{enumerate}
\end{lemma}

\begin{proof}
(i) Let $k\in \N$. Then $d_\tau(a\otimes 1_k)\leq (1-\frac{1}{k+1})d_{\tau}(b\otimes 1_{k+1})$ for all $\tau\in K$.
 Since $A$ has strict comparison by 2-quasitraces in $K$, $a\otimes 1_k\precsim b\otimes 1_{k+1}$, which in turn implies that $d_{\tau}(a)\leq (1+\frac 1 {k+1})d_{\tau}(b)$ for all  $\tau\in \QT(A)$. Letting $k\to \infty$ we get that $d_\tau(a)\leq d_\tau(b)$ for all 
 $\tau\in \QT(A)$, as desired.
 
 (ii) Let $\epsilon>0$. By a compactness argument as in the proof of Lemma \ref{epsilondelta}, we find that there exists $\delta>0$ such that $\tau((a-\epsilon)_+)\leq \tau((b-\delta)_+)$ for all $\tau\in K$ (see  \cite[Propositions 5.1 and 5.3]{ers}). Let $f_\epsilon,g,g_\delta\colon \QT(A)\to [0,\infty]$ be as follows: $f_\epsilon(\tau)=\tau((a-\epsilon)_+)$, $g(\tau)=\tau(b)$, and $g_\delta(\tau)=\tau((b-\delta)_+)$
for all $\tau$. The functions $f_\epsilon$ and $g$ (and also $g_\delta$) belong to the realification of $\Cu(A)$, as defined in \cite{riesz}. That is,  $f_\epsilon=f_n\uparrow$ and  $g=g_n\uparrow$, where  $f_n(\tau)=\frac{1}{r_n}d_\tau([a_n])$ and $g_n(\tau)=\frac{1}{s_n}d_\tau([b_n])$, for some $[a_n],[b_n]\in \Cu(A)$ and $r_n,s_n\in \N$ for $n=1,\dots$. (This follows from the fact that
\[
\tau(c)=\int_0^{\|c\|}d_\tau([(c-t)_+])dt
\] 
for all $c\in (A\otimes \mathcal  K)_+$ and $\tau\in \QT(A)$.)  By \cite[Propositions 5.1 and 5.3]{ers}, the function $g_\delta$ is way below $g$, so that $g_\delta\leq g_n$ for some $n$.  Hence, $f_m(\tau)\leq g_n(\tau)$ for all $\tau\in K$ and all $m\in \N$; 
i.e.,  $\frac{1}{r_m}d_\tau(a_m)\leq \frac{1}{s_n}d_\tau(b_n)$ for all
$\tau\in K$ and $m$.
By (i), this same inequality holds for all $\tau\in \QT(A)$; whence,  
 $f_m\leq g$ for all $m$. Passing to the supremum over $m$ we get
$ \tau(a-\epsilon)_+\leq \tau(b)$ for all $\tau\in \QT(A)$ and $\epsilon>0$. Now passing to the supremum over all $\epsilon>0$ we get  $\tau(a)\leq \tau(b)$ for all $\tau\in \QT(A)$, as desired.
\end{proof}

There is a version of the previous lemma for strict comparison of full positive elements by  traces:
\begin{lemma}\label{equalontracesbounded}
		Let $A$ be a $\Cstar$-algebra with $\mathrm{Prim} A$  compact and with strict comparison of full positive elements by traces. Let $a,b\in (A\otimes \mathcal K)_+$ be full positive elements.
		\begin{enumerate}[(i)]
			\item
			If	$d_\tau(a)\leq d_\tau(b)$ for all $\tau\in \Text(A)$ then $d_\tau(a)\leq d_\tau(b)$ for all $\tau\in \QT(A)$.
			\item
			 If  $\tau(a)\leq \tau(b)$ for all $\tau\in\Text(A)$ then $\tau(a)\leq \tau(b)$ for all $\tau\in \QT(A)$.
		\end{enumerate}
	\end{lemma}

Note: Since $b$ is full, the inequality $d_\tau(a)\leq d_\tau(b)$ need only be verified on densely finite traces, for otherwise $d_\tau(b)=\infty$.

\begin{proof}
The same proof as Lemma \ref{equalontraces}, with the obvious modifications, works here. When taking functional calculus cut-downs $(a-\epsilon)_+$ and $(b-\delta)_+$, we must take care to choose them so that they are still full (which is possible by the compactness of $\mathrm{Prim}(A)$).  
\end{proof}

\begin{theorem}\label{strictbytraces}
Let $A$ be a $\Cstar$-algebra.
\begin{enumerate}[(i)]
\item
If $A$ has strict comparison of positive elements by traces  then every lower semicontinuous 2-quasitrace on $A$ is a trace.

\item
If $Prim A$ is compact  and $A$ has strict comparison of full positive elements by  traces, then every densely finite lower semicontinuous 2-quasitrace on $A$ is a trace.
\end{enumerate}
\end{theorem}

\begin{proof}
(i)
Let $a,b\in A_+$. Let $c,d\in M_2(A)$ be defined as 
\[
c=\begin{pmatrix}
a+b & \\ & 0
\end{pmatrix},
d=\begin{pmatrix}
a & \\ & b
\end{pmatrix}.
\]
 Then $\tau(c)=\tau(d)$ for all $\tau\in \Text(A)$. By Lemma \ref{equalontraces} (ii), we get that $\tau(c)=\tau(d)$
for all $\tau\in \QT(A)$. But $\tau(c)=\tau(a+b)$ and $\tau(d)=\tau(a)+\tau(b)$.
So all $\tau\in \QT(A)$ are additive, as desired.   

(ii) The same proof as in  (i), but relying now on Lemma \ref{equalontracesbounded} (ii), shows in this case that the lower semicontinuous 2-quasitraces  on $A$ are additive on pairs of full positive elements. Let us now prove that the densely finite   ones are additive on any pair of positive elements. Let $\tau$ be one such   2-quasitrace   and let $a,b\in A_+$. 
Say $w\in A_+$ is full (whose existence is guaranteed by the compactness of $\mathrm{Prim}(A)$) and let $e_1,e_2,\dots$ be an approximate unit of $C^*(a,b,w)$ such that $e_{n+1}e_n=e_n$ for all $n$. Notice that $e_n$ is full
for large enough $n$ by the compactness of $\mathrm{Prim}(A)$. So
\begin{align*}
\tau(e_n a e_n + e_n b e_n)+2\tau(e_{n+1}) 
&=\tau(e_n a e_n + e_{n+1} + e_nbe_n + e_{n+1})\\
&=\tau(e_nae_n+e_{n+1})+\tau(e_nbe_n+e_{n+1})\\
&=\tau(e_n a e_n)+\tau(e_n b e_n)+2\tau(e_{n+1}).
\end{align*}
In the first and third equalities we have used the additivity of $\tau$ on commutative $\Cstar$-algebras and in the middle equality the additivity of $\tau$
on pairs of full elements.
 Since $\tau(e_{n+1})<\infty$, we get
\[
\tau(e_nae_n + e_nbe_n)=
\tau(e_nae_n)+\tau(e_n b e_n),
\] 
for all $n\in \N$. Letting $n\to \infty$ and using the lower semicontinuity of $\tau$ we obtain that $\tau(a+b)=\tau(a)+\tau(b)$, as desired.
\end{proof}

\begin{remark}\label{sameas}
In view of part (i) of the previous theorem, the property of strict comparison of positive elements by traces is equivalent to ``strict comparison by 2-quasitraces" and 
``2-quasitraces are traces" (all traces and 2-quasitraces are assumed  to be lower semicontinuous). Observe also that if $A$ has strict comparison of positive elements by 2-quasitraces and 
its densely finite 2-quasitraces are traces, then $A$ has strict comparison of full positive elements by traces.
Indeed,  if $b\in (A\otimes \mathcal K)_+$ is a full element the inequality $d_\tau(a)\leq (1-\epsilon)d_\tau(b)$ need only be verified on all densely finite    traces (otherwise  $\tau(b)=\infty$). 
\end{remark}

\section{Proof of Theorem \ref{mainpure}}\label{proofofmain}
In this section we prove Theorem \ref{mainpure} from the introduction. The proof is preceded by a number of preparatory results.

\begin{theorem}\label{product}
Let $A_1,A_2,\dots$  be $\Cstar$-algebras with strict comparison by traces.
Let $h\in \prod_{n=1}^\infty A_n$ be such that $h_n\sim_{\Tr}0$ for all $n\in \N$. Then $h\sim_{\Tr}0$.
\end{theorem}

\begin{proof}
Let $h\in \prod_{n=1}^\infty A_n$ be such that $h_n\sim_{\Tr}0$ for all $n$.
First, let us  show how to reduce ourselves to the case that the $\Cstar$-algebras $A_n$ are $\sigma$-unital for all $n\in \N$. Fix $n\in \N$. Since $h_n\in \overline{[A_n,A_n]}$, there exist finite sums of commutators 
$\sum_{i=1}^{N_k} [x_{i}^{(k)},y_{i}^{(k)}]$, with $x_{i}^{(k)},y_{i}^{(k)}\in A$ for all $i,k$, converging to $h_n$ as $k\to \infty$.  Consider the $\Cstar$-subalgebra $B_n=\her(\{x_{i}^{(k)},y_{i}^{(k)}\mid i=1,\dots,N_k, k\in \N\})$. Then $B_n$ has strict comparison by traces and  is  $\sigma$-unital. Furthermore, $h_n\in \overline{[B_n,B_n]}$. To prove the theorem, it suffices to show that   $h\sim_{\Tr}0$ in $\prod_{n=1}^\infty B_n$. Thus, from this point on,  we assume that the $\Cstar$-algebras $A_1,A_2,\dots$ are all $\sigma$-unital.

Let us set $\prod_{n=1}^\infty A_n=A$. As before, let $h\in A$ be such that $h_n\sim_{\Tr}0$ for all $n$. We may assume that  $\|h\|\leq 1$. 
Let us suppose, for the sake of contradiction, that  $\mu(h)\neq 0$ for some trace  $\mu$ on $A$ of norm $1$.  Notice then that $(h_n)_+\sim_{\Tr}(h_n)_-$ for all $n\in\N$, but 
$\mu( h_+)\neq \mu(h_-)$. We wish, however, to find  positive elements $a_n,b_n\in A_n$ agreeing on all traces in $\Text(A_n)$ (i.e., l.s.c. and $[0,\infty]$-valued)  while at the same time  $\mu((a_n)_n)\neq \mu((b_n)_n)$.
Let us  show how to  achieve this:  Assume, without loss of generality, that  $\mu(h) >0$. Set $\mu(h)=\delta$. Fix $n\in \N$. Let $(e_{n}^{(i)})_i$ be an approximate unit of $A_n$ such that 
$e_{n}^{(i+1)}e_{n}^{(i)}=e_{n}^{(i)}$ for all $i$. 
We can find  $h_{n}^{(i)}\in \her(e_{n}^{(i)})$ for $i$ large enough such that $\|h_{n}^{(i)}-h_n\|<\frac \delta 2$
and $h_{n}^{(i)}\sim_{\Tr} 0$ in $\her(e_{n}^{(i)})$. (This is achieved as follows: first, sufficiently approximate $h_n$ by a finite sum of commutators; next, multiply the elements in these commutators by $e_{n}^{(i)}$ and let $i\to \infty$.) Let 
\begin{align*}
a_n &:=(h_{n}^{(i)})_++e_{n}^{(i+1)},\\
b_n &:=(h_{n}^{(i)})_-+e_{n}^{(i+1)}. 
\end{align*}
Let $\tau \in \Text(A_n)$. Suppose first that
$\tau(e_{n}^{(i+1)})<\infty$. Then $\tau$  is bounded on $\her(e_n^{(i)})$, and so  $\tau(h_n^{(i)})=0$, because  $h_n^{(i)}\sim_{\Tr}0$   in $\her(e_n^{(i)})$. 
Hence,
\begin{align*}
\tau(a_n) &=\tau((h_{n}^{(i)})_+)+\tau(e_{n}^{(i+1)})\\
&=\tau((h_{n}^{(i})_-)+\tau(e_{n}^{(i+1)})\\
&=\tau(b_n).
\end{align*}
On the other hand, if $\tau(e_{n}^{(i+1)})=\infty$ then again we find that 
$\tau(a_n)=\infty=\tau(b_n)$. Furthermore,
\begin{align*}
\tau((a_n-t)_+) &=\tau( (h_{n}^{(i)})_+  + (e_{n}^{(i+1)}-t)_+ )\\
&=\tau((h_{n}^{(i)})_-+ (e_{n}^{(i+1)}-t)_+)\\
&=\tau((b_n-t)_+)
\end{align*}
for all $0\leq t<1$ and all $\tau\in \Text(A_n)$. (This equality is again verified both if $\tau((e_{n}^{(i+1)}-t)_+)<\infty$ and if $\tau((e_{n}^{(i+1)}-t)_+)=\infty$.)

Let $a=(a_n)_n$ and $b=(b_n)_n$. Notice that $\mu(a)-\mu(b)=\mu((h_{n}^{(i)})_n)>\delta/2$.
Let $K=\overline{\bigcup_{n=1}^\infty \Text(A_n)}$. Let us show that $\tau(a)=\tau(b)$ for all $\tau\in K$. Consider the set
\[
Q=\{\tau\in \Text(A)\mid \tau((a-t)_+)=\tau((b-t)_+)\hbox{ for all }0\leq t<1\}.
\]
Clearly, $\tau(a)=\tau(b)$ for all $\tau\in Q$. So it suffices to show  that $K\subseteq Q$.  By our construction of $a$ and $b$, we have $\bigcup_{n=1}^\infty \Text(A_n)\subseteq Q$. We will be done once we have shown that $Q$ is closed in $\Text(A)$. Suppose that $\tau_\lambda\to \tau$ in $\Text(A)$, with $\tau_\lambda\in Q$ for all $\lambda$.  Let $0\leq t<1$ and choose $t<t'<1$. Then
\[
\tau((a-t')_+)\leq \liminf \tau_\lambda((a-t')_+)\leq \limsup \tau_{\lambda}(b-t')_+\leq \tau(b-t)_+.
\]
Passing to the supremum over all $t'>t$ on the left, we get that $\tau((a-t)_+)\leq \tau((b-t)_+)$. By symmetry, we also have $\tau((b-t)_+)\leq \tau((a-t)_+)$.
Thus, $\tau\in Q$ as desired. 

Let us now show that $A$ has strict comparison of positive elements by $K$ (as defined in Definition \ref{strictbyK}). Let $c,d\in (A\otimes\mathcal K)+$  be such that
 $d_\tau(c)\leq (1-\gamma)d_\tau(d)$ for all $\tau\in K$ and  some $\gamma>0$. In order to show that $[c]\leq [d]$,
it suffices to show that $[(c-\epsilon)_+]\leq [d]$ for all $\epsilon>0$. But, for each $\epsilon>0$ and $\delta>0$, we have $(c-\epsilon)_+\sim_{\Cu}c'\in M_N(A)$ and $(d-\delta)_+\sim_{\Cu} d'\in M_N(A)$
for some $N>0$. Thus, applying Lemma \ref{epsilondelta}, we may reduce the proof to the case that $c,d\in M_N(A)$ for some $N\in\N$. Let us assume this.
Let us fix $\epsilon>0$. Again by Lemma \ref{epsilondelta}, there exists $\delta>0$ such that
\[
d_\tau((c-\epsilon)_+)\leq \Big(1-\frac \gamma 2\Big)d_\tau((d-\delta)_+),\hbox{ for all }\tau\in K.
\] 
Since $M_N(A)\cong \prod_{n} M_N(A_n)$, we can write $c=(c_n)_n$ and $d=(d_n)_n$, with $c_n,d_n\in M_N(A_n)$ for all $n$. Projecting onto $A_n$, we get
\[
d_\tau((c_n-\epsilon)_+)\leq \Big(1-\frac \gamma 2\Big)d_\tau((d_n-\delta)_+),\hbox{ for all }\tau\in \Text(A_n).
\] 
Since the $\Cstar$-algebra $A_n$ has strict comparison of positive elements by traces, we get that $[(c_n-\epsilon)_+]\leq [ (d_n-\delta)_+]$. Thus, $(c_n-2\epsilon)_+=x_n^*x_n$ and $x_nx_n^*\in \her((d_n-\delta)_+)$ for some $x_n\in A_n$.  Then
$(c-2\epsilon)_+=x^*x$ and $xx^*\leq Md$ for some $M>0$, where $x=(x_n)_n$. Hence, $[(c-2\epsilon)_+]\leq [d]$ for all $\epsilon>0$. Letting $\epsilon\to 0$, we get $[c]\leq [d]$, as desired.

We now know that $\tau(a)=\tau(b)$ for all $\tau\in K$ and that $A$ has strict comparison by $K$.
By  Lemma \ref{equalontraces} (ii), we conclude that
$\tau(a)=\tau(b)$ for all $\tau \in \Text(A)$. But this contradicts that 
  $\mu(a)\neq \mu(b)$, which completes the proof.
\end{proof}

Essentially the same proof, with some modifications, yields the following theorem:
\begin{theorem}\label{productbounded}
	Let $A_1,A_2,\dots$  be  $\Cstar$-algebras with $\mathrm{Prim}(A_n)$ compact and with strict comparison of full positive elements by traces.
	Let $h\in \prod_{n=1}^\infty A_n$ be such that $h_n\sim_{\Tr}0$ for all $n\in \N$. Then $h\sim_{\Tr}0$.	
\end{theorem}
\begin{proof}
	Let us sketch the necessary modifications to the proof of Theorem \ref{product} that yield 
	a proof of the present theorem: As before, we can reduce to the case that the $\Cstar$-algebras $A_1,A_2,\dots$ are $\sigma$-unital. To this end, we use  that $\mathrm{Prim} A_n$ is  if and only if there exist $c_n\in (A_n)_+$ and $\epsilon>0$ such that $(c_n-\epsilon)_+$ is full. Now defining $B_n=\her(\{c_n,x_i^{(k)},  y_i^{(k)}\mid i,k\})$, we guarantee that $B_n$ has compact primitive spectrum for all $n$ and is $\sigma$-unital. Next, 
the elements $a_n$ and $b_n$ in the first part of the proof are constructed as before, except that we take care that they be full elements. This is possible since the elements of the approximate unit $(e_n^{(i)})_i$ are full for large enough $i$. The definitions of the   sets $K$ and $Q$ remain unchanged, and again we find that $K\subseteq Q$ and that $Q$ is closed. In the next segment of the proof of Theorem \ref{product}  it is shown that $A$ has strict comparison of positive elements by $K$. The same arguments can be used to show that, in the present case,  the strict comparison by $K$ holds for \emph{full} positive elements; i.e., assuming that $c$ and $d$ are full. We finish the proof as before, now relying on Lemma \ref{equalontracesbounded},  rather than Lemma \ref{equalontraces}.
\end{proof}

We deduce from the previous theorems the following corollaries:
\begin{corollary}\label{halfnorm}
There exists $N\in \N$ such that if $A$ is a $\Cstar$-algebra with strict comparison of positive elements by traces and $h\in A$ is such that
$h\sim_{\Tr}0$, then
\[
\|h-\sum_{i=1}^N [x_i,y_i]\|\leq \frac 1 2\|h\|
\] 
for some $x_i,y_i\in A$ such that 
$\|x_i\|\cdot \|y_i\|\leq \|h\|$ for all $i$. 
\end{corollary}
\begin{proof}
Let us suppose for the sake of contradiction that  no such $N$ exists. Then there exist $\Cstar$-algebras $A_1,A_2,A_3,\dots$ with strict comparison by traces and contractions $h_n\in A_n$ such that $h_n\sim_{\Tr} 0$ and the distance from $h_n$ to elements of the form $\sum_{i=1}^n [u_i,v_i]$, with $\|u_i\|,\|v_i\|\leq 1$ for all $i$, is at least $1/2$ for all $n\in \N$. Let $h=(h_n)\in \prod_{n=1}^\infty A_n$. By Theorem \ref{product}, $h\sim_{\Tr} 0$. Hence,
$\|h-\sum_{j=1}^N[x_j,y_j]\|<1/2$ for some $x_j,y_j\in \prod_{n=1}^\infty A_n$. Increasing $N$ if necessary, we may assume that $\|x_j\|,\|y_j\|\leq 1$ for all $j=1,\dots,N$.
We get a contradiction projecting onto  the $N$-th coordinate.
\end{proof}

 The same proof, now relying on Theorem \ref{productbounded}, yields the following 
 corollary:

\begin{corollary}\label{halfnormbounded}
	There exists $N'\in \N$ such that if $A$ is  a $\Cstar$-algebra with $\mathrm{Prim}(A)$ compact and  with strict comparison of full positive elements by  traces, and $h\in A$ is such that
	$h\sim_{\Tr}0$, then
	\[
	\|h-\sum_{i=1}^{N'} [x_i,y_i]\|\leq \frac 1 2\|h\|
	\] 
	for some $x_i,y_i\in A$ such that 
	$\|x_i\|\cdot \|y_i\|\leq \|h\|$ for all $i$. 
\end{corollary}

From these corollaries we deduce  the theorem on traces of products
and ultraproducts stated in the introduction: 
\begin{proof}[Proof of Theorem \ref{mainproducts}]
Imitate the proof of  \cite[Theorem 8]{ozawa}, now relying on Corollary \ref{halfnorm} or Corollary \ref{halfnormbounded} instead of on  
\cite[Theorem 6]{ozawa}.
\end{proof}


In order to obtain finite commutator bounds for a pure $\Cstar$-algebra whose 2-quasitraces are traces, we intend to apply Theorem \ref{fackstechnique}. We have already shown that condition (ii) of that theorem is met by this class of $\Cstar$-algebras (in Corollary \ref{halfnorm}). In the next lemmas we establish the existence of a sequence of pairwise orthogonal positive elements as in Theorem \ref{fackstechnique} (i). 

Recall that, given positive elements $a$ and $b$, by $b\precsim a$ we mean that $b=x^*x$ and $xx^*\in \her(a)$ for some $x\in A$.
\begin{lemma}\label{strongcomp}
Let $A$ be a pure $\Cstar$-algebra and $a,b\in (A\otimes \mathcal K)_+$. If $d_\tau(b)\leq \gamma d_\tau(a)$ for all $\tau\in \QT(A)$ 
and some $\gamma<1/2$ then $b\precsim a$. 
\end{lemma}
\begin{proof}
The proof follows closely that of \cite[Theorem 4.4.1]{radius}, but we take care to remove the assumption on finite quotients needed there. First, using functional calculus,
let us    find $b_i,b_i',b_i''\in C^*(b)_+$, with $i=1,2,\dots$, such that 
\begin{enumerate}
	\item
	$b_i'b_i=b_i$ and  $b_i''b_i'=b_i'$ for all $i$, 
	\item
	$b_i''\perp b_j''$ for all $i,j$ such that $i\neq j$ and  $i-j$ is even, 
	\item
	 $b=\sum_{i=1}^\infty b_i$. 
\end{enumerate}
Since $b\precsim \bigoplus_{i=1}^\infty \frac 1 i b_i$, it suffices to show that $\bigoplus_{i=1}^\infty \frac 1 i b_i\precsim a$. Let us prove this.

We have
\[
\sum_{i=1}^\infty d_\tau(b_i'')\leq \sum_{i=1}^\infty d_\tau(b_{2i}'')+\sum_{i=1}^\infty d_\tau(b_{2i-1}'')
\leq 2d_\tau(b)\leq 2\gamma d_\tau(a).
\]
Since $A$ is pure, we can choose $d\in (A\otimes \mathcal K)_+$
 such that $d_\tau(b_1''\oplus d)\leq \gamma_1d_\tau(a)$ for some $2\gamma<\gamma_1<1$, and $[b_1'']\leq k [d]$ for some $k\in \N$. (E.g., choose $d$ such that  $N[d]\leq [a]\leq (N+1)[d]$ with $N$ large enough.) It follows by the strict comparison property of $A$ that $[b_1''\oplus d]\leq [a]$. Since $[b_1']\ll[b_1'']$ (where $\ll$ is the relation of compact containment in $\Cu(A)$), there exists    $\epsilon>0$ such that  $[b_1']\leq  k[(d-\epsilon)_+]$. Let $d'=(d-\epsilon)_+$.  We then have that $b_1'\oplus d'\precsim (a-\delta)_+$ for some $\delta>0$.
 Let $v\in (A\otimes \mathcal K)^{**}$ be a partial isometry implementing this subequivalence.
 Let $c_1',e\in \her((a-\delta)_+)$ be given by $c_1'=vb_1'v^*$ and  $e=vd'v^*$.
Let  $g_\delta(a)\in C^*(a)_+$ be strictly positive and such that $g_\delta(a)(a-\delta)_+=(a-\delta)_+$.
Finally, set $a_1=g_\delta(a)-c_1'$. Then
\[
\sum_{i=2}^\infty d_\tau(b_i'')+d_\tau(c_1')\leq 2\gamma d_\tau(a)\leq 2\gamma d_\tau(a_1)+d_\tau(c_1'),
\]
for all $\tau\in \QT(A)$. If $d_\tau(c_1')<\infty$, we get
\begin{align}\label{cancelc1}
\sum_{i=2}^\infty d_\tau(b_i'')\leq 2\gamma d_\tau(a_1),
\end{align}
Suppose that $d_\tau(c_1')=\infty$. Then $d_\tau(b_1')=d_\tau(c_1')=\infty$. Since $[b_1']\leq k [d']$, we also have $d_\tau(e)=d_\tau(d')=\infty$. But $a_1e=(g_\delta(a)-c_1')e=e$. Hence, $d_\tau(a_1)=\infty$. So again we have \eqref{cancelc1}. 
Let $c_1=vb_1v^*$ and notice that $a_1\perp c_1$.
We can repeat the same arguments,  now finding positive elements $c_2,a_2\in \her(a_1)$ such that $b_2\sim c_2$,  $a_2\perp c_2$ 
and
$\sum_{i=3}^\infty d_\tau(b_i'')\leq 2\gamma d_\tau(a_2)$. Continuing this process ad infinitum, we obtain  $c_1,c_2,\dots\in \her(a)$ such that  $b_i\sim c_i$ for all $i$ and $c_i\perp c_j$ for all $i\neq j$. Hence, $\bigoplus_{i=1}^\infty \frac 1 i b_i\sim \sum_{i=1}^\infty \frac 1 i c_i\in \her(a)$, which proves the lemma.
\end{proof}

The previous lemma implies that in a pure $\Cstar$-algebra the ordered semigroup $W(A)$ is hereditary in $\Cu(A)$. This will not be needed later on but  has independent interest. Recall that  $\mathrm W(A)$  is defined as 
\[
\mathrm W(A)=\{[a]\in \Cu(A)\mid a\in \bigcup_{n=1}^\infty M_n(A)\}.
\]
\begin{corollary}
Let $A$ be a pure $\Cstar$-algebra. The ordered semigroup $\mathrm W(A)$ is a hereditary (in the order-theoretic sense) subsemigroup of $\Cu(A)$.
\end{corollary}
\begin{proof}
Let $a\in A\otimes \mathcal K$ and $b\in M_n(A)$ be positive elements such that $[a]\leq [b]$. By Lemma \ref{strongcomp}, $a\precsim b^{\oplus 3}\in M_{3n}(B)$. Hence $a\sim a'\in \her(b^{\oplus 3})\subseteq M_{3n}(B)$, which in turn  implies that  $[a]=[a']\in \mathrm W(A)$.
\end{proof}

In the following lemma we make use of the abundance of \emph{soft} elements in the Cuntz semigroup of a pure $\Cstar$-algebra (see    \cite{antoine-perera-thiel}). It will not be necessary here to   recall their definition and multiple properties. We will merely need the following fact: 
\emph{Let $A$ be a pure $\Cstar$-algebra and $[a]\in \Cu(A)$. Then there exists $[a]_s\in \Cu(A)$  such that 
		$[a]_s\leq [a]$, $d_\tau([a]_s)=d_\tau([a])$ for all $\tau\in \QT(A)$, and
		$[a]_s$ is exactly divisible by all $n\in \N$; i.e., for each $n\in \N$ there exists $[b]\in \Cu(A)$ such that $n[b]=[a]_s$.} 
The proof of this fact can be extracted from \cite{antoine-perera-thiel} as follows:
	By \cite[Theorem 7.3.11]{antoine-perera-thiel}, the Cuntz semigroup  of a pure $\Cstar$-algebra has $Z$-multiplication (in the sense of  \cite[Definition 7.1.3]{antoine-perera-thiel}). Here $Z$ denotes the Cuntz semigroup of the Jiang-Su $\Cstar$-algebra $\mathcal Z$.
	Then $[a]_s=1'\cdot [a]$, with $1'\in Z$ denoting the ``soft" 1, has the desired properties. See \cite[Proposition 7.3.16]{antoine-perera-thiel}.

\begin{lemma}\label{geomseries}
Let $A$ be a $\sigma$-unital pure $\Cstar$-algebra and $e_0\in A_+$ a  strictly positive element. Then there exist pairwise orthogonal positive elements $e_1,e_2,\dots \in A_+$ such that
$e_{i}\precsim e_{i+1}^{\oplus 5}$ for all $i\geq 1$ and $e_0\precsim e_1^{\oplus 11}$.  
\end{lemma}
\begin{proof}
Let $[e_0]_s\leq [e_0]$ be the soft element associated to the Cuntz semigroup element $[e_0]$.   
Let us find $[f_1],[f_2],\dots\in \Cu(A)$ such that $5[f_1]=[e_0]_s$ and $2[f_{i+1}]=[f_i]$ for   all $i\geq 1$. Let $f\in (A\otimes \mathcal K)_+$ be given by $f=f_1\oplus \frac 1 2f_2\oplus \frac 1 3 f_3\oplus\cdots $.
Then
\begin{align*}
 d_\tau(f_i) &= \frac{2}{5}d_\tau(f_{i+1}^{\oplus 5}),\\
 d_\tau(a) &= \frac{5}{11}d_\tau(f_1^{\oplus 11}),\\
 d_\tau(f) &= \frac{2}{5}d_\tau(e_0),
\end{align*}
for all $\tau\in \QT(A)$.  Hence, by Lemma \ref{strongcomp},
$f_i\precsim f_{i+1}^{\oplus 5}$, $e_0\precsim f_1^{\oplus 11}$, and $f\precsim e_0$.  Let $v\in (A\otimes \mathcal K)^{**}$ be the partial isometry implementing the comparison $f\precsim a$. Then, the positive elements  $e_i=vf_iv^*$, with $i=1,2,\dots$, have the desired properties.  
\end{proof}


Let us now prove Theorem \ref{mainpure} from the introduction.

\begin{proof}[Proof of Theorem \ref{mainpure}]
	We  first prove that every $\Cstar$-algebra $A$ as in the theorem has finite commutator bounds  with no approximations. We then reduce the number of commutators to 7.

	Let $h\in \overline{[A,A]}$.  Then $\sum_{i=1}^{k_n} [x_i^{(n)},y_i^{(n)}]\to h$ for some $x_i^{(n)},y_i^{(n)}\in A$.
	Passing to the hereditary $\Cstar$-subalgebra $\her(\{x_i^{(n)},y_i^{(n)}\mid i=1,\dots,k_n,\, n=1,\dots\})$ if necessary, 
	we may assume that $A$ is $\sigma$-unital (since hereditary subalgebras  of pure $\Cstar$-algebras are again pure). Let $e_0\in A_+$ be a  strictly positive element.
By	Lemma \ref{geomseries},  $A$ contains a sequence  of pairwise orthogonal positive elements $(e_i)_{i=1}^\infty$ such that $e_i\precsim e_{i+1}^{\oplus 7}$ for all $i$ and $e_0\precsim e_1^{\oplus 11}$. Furthermore, Corollary \ref{halfnorm} is applicable to every hereditary subalgebra of $A$. Thus, Theorem \ref{fackstechnique} is applicable to $A$.
That is, $A$ has finite commutator bounds $(n,C)$ with no approximations, for some $n\in \N$ and $C>0$. 

Let us now reduce the number of commutators to 7. Since $A$ is pure, we can find $b\in (A\otimes \mathcal K)_+$
such that $(2n+1)[b]\leq [e_0]\leq (2n+2)[b]$. Then 
\begin{align*}
	d_\tau(b^{\oplus n}) \leq \frac{n}{2n+1}d_\tau(e_0)\hbox{ and }
	d_\tau(e_0) \leq \frac{2n+2}{5n}d_\tau(b^{\oplus 5n}).
\end{align*}
Hence, by Lemma \ref{strongcomp}, there exists  $f\in A_+$ such that $b^{\oplus n}\sim f$ and $e_0\precsim f^{\oplus 5}$. 
Now let $h\in A$ be such that $h\sim_{\Tr}0$.  From Lemma \ref{aplusnb} we obtain that 
$h=\sum_{i=1}^5 [x_i,y_i]+h'$, for some  $h'\in \her(f)$ such that $h'\sim_{\Tr}0$ in $\her(f)$.  But $\her(f)\cong M_n(\her(b))$. Thus, by Theorem \ref{matrixreduction} (ii),
$h'=[x_6,y_6]+[x_7,y_7]$. Furthermore, the sum 
$\sum_{i=1}^7 \|x_i\|\cdot \|y_i\|$ is bounded by $C'\|h\|$ for some $C'>0$, as can be seen from  the statements of Lemma \ref{aplusnb} and Theorem \ref{matrixreduction}.
\end{proof}


In Theorem \ref{mainpure}, it is possible to reduce further the number of commutators  under a variety of additional assumptions.
We  show in Theorem \ref{reductions} below that if the $\Cstar$-algebra is  assumed to be unital, then three commutators suffice. We need a couple of lemmas.

\begin{lemma}\label{ddroplemma}
Let $d\in M(A)$ be a multiplier positive contraction 
satisfying that $d\oplus d\oplus d\precsim 1-d$ in $M(A)$. Then for each $h\in A$ there exist 
$x,y\in A$ and $h'\in \her(1-d)$ such that $h=[x,y]+h'$ and $\|x\|\cdot \|y\|\leq 2\|h\|$.
\end{lemma}

\begin{proof}
We have $h=dh+(1-d)hd+(1-d)h(1-d)$. Let 
\begin{align*}
	f&=dh+(1-d)hd,\\
	e&=d+hd^2h^*+(1-d)hd^2h^*(1-d). 
	\end{align*}
Then $f\in \her(e)$ and $e\precsim d\oplus d\oplus d\precsim 1-d$. So we can apply  Lemma \ref{aplusnb}  with $n=1$ to $f$. We get $f=[x,y]+f'$, with $f'\in \her(1-d)$ and $\|x\|\cdot \|y\|\leq \|f\|\leq 2\|h\|$. Hence,  $h=[x,y]+(1-d)h(1-d)+f'$, which proves the lemma. 
\end{proof}

Let $\mathcal Z_{n-1,n}$ denote the dimension drop $\Cstar$-algebra:
\[
\mathcal Z_{n-1,n}=\{f\in M_{(n-1)n}(C[0,1])\mid f(0)\in M_{n-1}\otimes 1_{n},f(1)\in M_{n}\otimes 1_{n-1}\}.
\]

\begin{proposition}\label{ddembeds}
	Let $A$ be a pure unital $\Cstar$-algebra. Then for all $n\in \N$ there exists a unital homomorphism $\phi\colon \mathcal Z_{n-1,n}\to A$.
\end{proposition}
\begin{proof}
Let  $n\in \N$.	Since $\Cu(A)$ is almost divisible, we can find $[a]\in \Cu(A)$ such that $n[a]\leq [1]$ and  $[1]\leq (n+1)[(a-\epsilon)_+]$ for some $\epsilon>0$.
	In turn, this implies that there exist  pairwise orthogonal positive elements $b_1,b_2,\dots,b_n\in A_+$ such that $(a-\frac \epsilon 2)_+\sim b_1\sim\dots\sim b_n$.
	Using functional calculus, let us modify $b_i$ so that there exists $b_i'\in \her(b_i)_+$, with $b_ib_i'=b_i'$ and $b_i'\sim (a-\epsilon)_+$ for all $i$.
	Set $\sum_{i=1}^n b_i=b$. Then
	$(1-b)b_i'=(1-b_i)b_i'=0$ for all $i$. That is, $1-b\perp b_i'$ for all $i$. We have 
	\begin{align*}
		d_\tau(1-b)+nd_\tau(b_1') &=d_\tau(1-b)+d_\tau\Big(\sum_{i=1}^n b_i'\Big)\\
		&\leq d_\tau(1)\\
		&\leq (n+1)d_\tau(b_1'),
	\end{align*}
	for all $\tau\in \QT(A)$.
	Hence, $d_\tau(1-b)\leq d_\tau(b_1')=\frac 1 3 d_\tau(b_1'\oplus b_2'\oplus b_3')$.
	By Lemma \ref{strongcomp}, this implies that $1-b\precsim b_1'\oplus b_2'\oplus b_3'$.
	Let us assume now that $n=3k$, for some $k\in \N$. By \cite[Proposition 5.1]{rordam-winter},  there exists a unital homomorphism from  the dimension drop $\Cstar$-algebra $\mathcal Z_{k,k+1}$  into $A$. 
\end{proof}

\begin{theorem}\label{reductions}
	Let $A$ be a pure  $\Cstar$-algebra with compact $\mathrm{Prim}(A)$ and with strict comparison of full positive elements  by traces.
	Then $A$ has commutator bounds $(7,C_1)$, with no approximations, for some universal constant $C_1$. If $A$ is unital then it has commutator bounds
$(3,C_2)$, with no approximations, for some universal constant $C_2$.

\end{theorem}
\begin{proof}
Let us  first show that   $A$    has finite commutator bounds.  To this end, we   proceed as in the proof of Theorem \ref{mainpure}, but with a few small modifications. (The main difference with  Theorem \ref{mainpure} being that now we only require strict comparison by traces on full positive elements.)  Lemma \ref{geomseries} is applicable to $A$, yielding a  sequence of pairwise orthogonal positive elements $(e_i)_{i=1}^\infty$ such that $e_0\precsim e_1^{\oplus 11}$ and $e_i\precsim e_{i+1}^{\oplus 7}$ for all $i$. Here $e_0\in A_+$ is strictly positive. The hereditary subalgebras 
$\her(e_i)$ have compact primitive spectrum for all $i$ (since they are full in $B$). Hence, Corollary \ref{halfnormbounded} is applicable in each of them.
Now Theorem \ref{fackstechnique}  implies that $A$ has finite commutator bounds $(n,C)$ with no approximations.  

The arguments for reducing  the number of commutators to 7 in the proof of Theorem \ref{mainpure} apply here as well.

Let us now show that  if $A$ is unital  then the number of commutators can be reduced to 3.  By Proposition \ref{ddembeds},  the dimension drop $\Cstar$-algebra $\mathcal Z_{n,n+1}$ maps  unitally into $A$.  By \cite[Proposition 5.1]{rordam-winter}, there exist $f_1,\dots,f_n\in A_+$ such that $f_1\sim f_i$ for all $i$ and $1-\sum_{i=1}^n f_i\precsim (f_1-\epsilon)_+$ for some $\epsilon>0$.
Let us assume without loss of generality that $n\geq 3$. By Lemma \ref{ddroplemma}, for each  $h\in \overline{[A,A]}$ we have $h=[x_1,y_1]+h'$, with $h'\in \her(f_1,\dots,f_n)$.
But $\her(f_1,\dots,f_n)\cong M_n(\her(f_1))$, and $\her(f_1)$ has commutator bounds $(n,C)$ with no approximations. It follows by Theorem \ref{matrixreduction} that 
$h'=[x_2,y_2]+[x_3,y_3]$. 
\end{proof}

%

\begin{example}\label{freeproducts}
Let $(A_i,\tau_i)$, with $i=1,2,\dots$, be unital $\Cstar$-algebras with   faithful tracial states. Assume that 
for infinitely many indices $i_1,i_2,\dots$ there exist unitaries $u_{i_n}\in A_{i_n}$ such that 
$\tau_{i_n}(u_{i_n})=0$ for all $n$. Let $A=A_1*A_2*\cdots$ and $\tau=\tau_1*\tau_2*\cdots$ be the reduced free product $\Cstar$-algebra and tracial state. 
It is known that $A$ is simple  and $\tau$ is its  unique
tracial state  (by \cite{avitzour}). Furthermore, by \cite[Proposition 6.3.2]{nccw}, $A$ has  strict comparison of positive elements by the trace  $\tau$.
It follows that $A$ is pure and, by Theorem \ref{strictbytraces}, that the only bounded 2-quasitraces on $A$ are the scalar multiples of $\tau$.
By  Theorem \ref{reductions}, if  $h\in A$ is such that $\tau(h)=0$  then $h$ is a sum of $3$ commutators.
\end{example}

\section{Sums of nilpotents of order 2}\label{nilpotents}
Let $A$ be a $\Cstar$-algebra. Let $N_2=\{x\in A\mid x^2=0\}$; i.e., $N_2$ denotes the set of   nilpotent elements  of order 2 in $A$ (a.k.a, square zero elements). 

\begin{lemma}\label{nilascomm}
Let $z\in N_2$. Then $z=[u,v]$ and $z+z^*=[w^*,w]$ for some $u,v,w\in A$.
\end{lemma}
 \begin{proof}
 We may assume that $\|z\|=1$. The universal $\Cstar$-algebra generated by  a square zero contraction is $M_2(C_0(0,1])$. Thus,  there exists  a homomorphism 
 $\phi\colon M_2(C_0(0,1])\to A$ such that $\tilde z:=(\begin{smallmatrix} 0 & t\\ 0 & 0\end{smallmatrix})\stackrel{\phi}{\mapsto}z$. So it suffices to express 
 $\tilde z$ and $\tilde z+\tilde z^*$ as  commutators. Indeed,
 \begin{align*}
 \begin{pmatrix}
 0 & t\\
 0 & 0
 \end{pmatrix} &=\Big[\begin{pmatrix}
 t^{\frac 1 2} & 0\\
 0 & 0
 \end{pmatrix},
 \begin{pmatrix}
 0 & t^{\frac 1 2}\\
 0 & 0
 \end{pmatrix}\Big],\\
 \begin{pmatrix}
 0 & t\\
 t & 0
 \end{pmatrix} &=\frac 1 4\Big[\begin{pmatrix}
 t^{\frac 1 2} & -t^{\frac 1 2}\\
 t^{\frac 1 2} & -t^{\frac 1 2}
 \end{pmatrix},
 \begin{pmatrix}
 t^{\frac 1 2} & t^{\frac 1 2}\\
 -t^{\frac 1 2} & -t^{\frac 1 2}
 \end{pmatrix}\Big].\qedhere
 \end{align*}
 \end{proof}
 
The preceding lemma implies that  linear span of $N_2$ is contained in $[A,A]$. In Theorem \ref{commutatornil} below we show, conversely, that if $A$ is a pure unital $\Cstar$-algebra then every commutator is expressible as a  sum of at most $14\times 256$ order 2 nilpotents.

\begin{lemma}\label{matricesnilpotents}
Let $A$ be a $\Cstar$-algebra and $a,b\in A$. 
\begin{enumerate}[(i)]
    \item If $a,b\in N_2$ then $[a,b]$ is a sum of 3 order 2 nilpotents.
    \item If $A=M_2(B)$ or $A=M_3(B)$ then $[a,b]$ is at most a  sum of 14 order 2 nilpotents.
\end{enumerate}
In both cases, the norm of the nilpotents is bounded by $C\|a\|\cdot \|b\|$, where $C$ is a universal constant.
\end{lemma}
\begin{proof}
(i) Let $a,b\in N_2$. Normalizing $a$ and $b$ if necessary we may assume  that they are contractions. Then, as pointed out in \cite[Lemma 3.2]{marcoux02}, 
\[ 
[a,b] = (ab + aba -b - ba) + (-aba) + (b),
\] 
and each term on the right hand side is an order 2 nilpotent of norm at most 4. 

(ii) This is proven by  Marcoux in  \cite[Theorem 5.6 (ii)]{marcoux02}, for $n=2$,  and in \cite[Theorem 3.5(ii)]{marcoux02} for $n=3$ (see also remarks after \cite[Theorem 5.1]{marcoux06}). Although in the statements of these theorems Marcoux assumes that $B$ is unital, a quick  inspection of the proofs reveals that this is not used.
\end{proof}

\begin{theorem}\label{commutatornil}
Let $A$ be a pure unital $\Cstar$-algebra and $a,b\in A$.
\begin{enumerate}[(i)]
    \item Then $[a,b]$ is a sum of at most $14\times 256$ nilpotents of order 2. The norm of the order 2 nilpotents is bounded by $C\|a\|\cdot \|b\|$, where
$C$ is a universal constant.
    \item If $[a,b]$ is selfadjoint then it is a sum of $14\times 256$   commutators of the form $[x^*,x]$ with $x\in N_2$, with $\|x\|^2\leq C'\|a\|\cdot \|b\|$ and $C'$ is a universal constant.
\end{enumerate}
\end{theorem}
\begin{proof}
(i)	
Let us choose $s_1$ and $s_2$ such that $0<s_1<s_2<1$. Let $f_1,f_2,f_3,f_4\in C([0,1])_+$ be functions such that 
$f_1$ is supported on $[0,s_1)$, $f_2$ is supported on $(0,s_2)$, $f_3$ is supported on
$(s_1,1)$, $f_4$ is supported on $(s_2,1]$, and $f_1+f_2+f_3+f_4=1$.  Let us regard $C([0,1])$
embedded in $\mathcal Z_{2,3}$ via the map $f\mapsto f\cdot 1_{6}$. Further, by Proposition \ref{ddembeds}, $\mathcal Z_{2,3}$ maps unitally into $A$. 
In this way, we can view $f_1,f_2,f_3,f_4$ as elements of $A$. Then\[
[a,b]=\sum_{1\leq i,j,k,l\leq 4}[f_iaf_j,f_kbf_l].
\]We will show that each of the $256$ terms on the right hand side is expressible as a sum of at most 14 nilpotents of order 2. 

Let us examine the commutator $[f_iaf_j,f_kbf_l]$.
Since $\her_A(f_1+f_2+f_3)\cong M_2(B)$ and $\her_A(f_2+f_3+f_4)\cong M_3(B')$ for some $B,B'\subseteq A$,
 if not all four functions appear in $[f_iaf_j,f_kbf_l]$ then  by Lemma \ref{matricesnilpotents}  this commutator is expressible as a sum of at most 14 order 2 nilpotents.

Let us examine the commutators $[f_iaf_j,f_kbf_l]$ where all four functions appear; i.e., such that $\{i,j,k,l\}=\{1,2,3,4\}$.
Let us assume that $i=1$. If $l=3$ or $l=4$ then $[f_iaf_j,f_kbf_l]$ is itself an order 2 nilpotent and we are done. Suppose  that $l=2$. Since we are assuming that all four indices must appear,  either $k=4$ and $j=3$, or $k=3$ and $j=4$. Suppose that $k=4$ and $j=3$. Since $f_1af_3$ and $f_4bf_2$  are both order 2 nilpotents,  $[f_1af_3,f_4bf_2]$
is a sum of three order 2 nilpotents by  Lemma \ref{matricesnilpotents} (i). Suppose now that $k=3$ and $j=4$. The commutator is then $[f_1af_4,f_3bf_2]$, which  can be dealt with as follows:
\begin{multline*}
[f_1af_4,f_3bf_2] =[f_1a(f_4f_3)^{\frac 1 2},(f_4f_3)^{\frac 1 2}bf_2]\\
+[(f_4f_3)^{\frac 1 2}b(f_2f_1)^{\frac 1 2},(f_2f_1)^{\frac 1 2}a(f_4f_3)^{\frac 1 2}]\\
+[(f_2f_1)^{\frac 1 2}af_4,f_3b(f_2f_1)^{\frac 1 2}].
\end{multline*}
Each of the commutators on the right side is a commutator of order 2 nilpotents  and is thus expressible as a sum of three order 2 nilpotents. So, $[f_1af_4,f_3bf_2]$ is expressible as a sum of 9 nilpotents of order 2.

Let us assume now that $i\neq 1$. As argued above, we may reduce ourselves to the case that one of the indices $j,k,l$ is 1. On the grounds of the symmetry of our set-up,  any of these cases  can be dealt with just as we did above for $i=1$.
(Notice that only the orthogonality relations between the functions were used in our analysis; the asymmetry of the dimension drop $\Cstar$-algebra played not role.) We are thus done.

(ii) Suppose that $[a,b]$ is selfadjoint. By (i), it can be written   as a sum with $14\times 256$ terms, each of the form  $z+z^*$, with $z\in N_2$. In turn, each of these terms is expressible as a  commutator of the form $[x^*,x]$, with $x\in N_2$, by Lemma  \ref{nilascomm}.
\end{proof}

\begin{theorem}\label{nilcommutatorsthm}
Let $A$ be a pure unital $\Cstar$-algebra whose bounded 2-quasitraces are traces. Then
the sets  $\overline{[A,A]}$, $[A,A]$, and the linear span of $N_2$, are equal. Each $h\in A$ such that $h\sim_{\Tr}0$ is expressible as a sum of 
$3\times 14\times 256$ square zero elements. If $h$ is selfadjoint then it is also expressible as 
sum of $3\times 14\times 256$ commutators of the form $[x^*,x]$, with $x\in N_2$. 
\end{theorem}

\begin{proof}
The assumptions on $A$ imply that it has strict comparison of full positive elements by traces (see Remark \ref{sameas}).
By   Theorem  \ref{reductions}, each $h\in \overline{[A,A]}$ is  expressible as a  sum  of $3$ commutators.
Each of these commutators, in turn, is expressible as a sum of $14\times 256$ square zero elements.  If $h$ is selfadjoint,
then it also expressible as a sum of  $14\times 256$ terms of  the form $z+z^*$, with $z\in N_2$, and each of these is a commutator $[x^*,x]$
with $x\in N_2$.
\end{proof}

\section{The kernel of the determinant map}\label{thedeterminant}
Let us briefly recall  the definition  of the de la Harpe-Skandalis determinant, as defined in \cite{dlHarpe-Skandalis1}. Let $A$ be a $\Cstar$-algebra.
Let $\mathrm{GL}_{\infty}(A)$ denote the infinite general linear group of $A$ and 
$\mathrm{GL}_{\infty}^0(A)$ the connected component of the  identity.
(If $A$ is non-unital, the general linear groups $\mathrm{GL}_n(A)$ are  defined as the subgroup of $\mathrm{GL}_n(A^\sim)$
of elements of the form $1_n+x$, with $x\in M_n(A)$, and $\mathrm{GL}_\infty(A)$ is the direct limit.)
Let $x\in \mathrm{GL}_{\infty}^0(A)$. Let $\eta\colon [0,1]\to \mathrm{GL}_{\infty}^0(A)$
be a path such that $\eta(0)=1$ and $\eta(1)=x$. Let
\[
\widetilde{\Delta}_\eta(x)=\frac{1}{2\pi i}\Tr\Big(\int_0^1\eta'(t)\eta(t)^{-1}\Big)\in A/\overline{[A,A]}.
\]
If $\zeta$ is another path connecting $1$ to $x$, then $\widetilde \Delta_\eta(x)-\widetilde \Delta_{\zeta}(x)$ belongs to the (additive) subgroup  $\{\Tr(p)-\Tr(q)\mid,p,q\hbox{ projections in }M_\infty(A)\}$, which we denote by $\underline{\Tr}(\mathrm K_0(A))$. The de la Harpe-Skandalis determinant  $\Delta_{\Tr}(x)$ is defined as  the image of $\widetilde \Delta_\eta(x)$
in the quotient of $A/\overline{[A,A]}$ by $\underline{\Tr}(\mathrm K_0(A))$. 
If $x=\prod_{k=1}^n e^{ih_k}$ then $\Delta_{\Tr}(x)$ can be computed to be the image of $\sum_{k=1}^n h_k$
in this quotient.

Let $\mathrm{U}(A)$ denote the unitary group of $A$ and $\mathrm{U}^0(A)$ the connected component of the identity. (If $A$ is non-unital, $\mathrm{U}(A)$ is defined as the  unitaries in $\mathrm U(A^\sim)$ of the form $1+x$, with $x\in A$.)
Given unitaries $u,v\in \mathrm U^0(A)$, let us denote by $(u,v)$  the multiplicative commutator
$uvu^{-1}v^{-1}$. 
Let $\mathrm{DU}^0(A)$ denote the derived or commutator subgroup of $\mathrm U^0(A)$. It is clear that 
$\mathrm{DU}^0(A)$ is contained in the kernel of $\Delta_\Tr$ (since $\Delta_{\Tr}$ is a group homomorphism with abelian codomain).
In this section we prove the following theorem:

\begin{theorem}\label{determinant}
Let $A=M_3(B)$, where $B$ is a unital pure $\Cstar$-algebra whose bounded 2-quasitraces are traces.  Then $\ker \Delta_{\Tr}\cap \mathrm{U}^0(A)= \mathrm{DU}^0(A)$.
\end{theorem}

The proof is preceded by a number of lemmas.

\begin{lemma}
 Let $A$ be a  pure unital $\Cstar$-algebra and $p\in M_m(A)$ a projection. 
Then there exists $h\in A_{\sa}$ such that $h\sim_{\Tr}p$ and $e^{ih}=(u,v)$
for some unitaries $u,v\in\mathrm {U}^0(A)$.
\end{lemma}

\begin{proof}
Let $B=pM_m(A)p$. Choose $n\in \N$. Since $B$ is pure and unital, there exists 
a unital homomorphism $\phi\colon \mathcal Z_{n-1,n}\to B$, by Proposition \ref{ddembeds}. Let $e\in \mathcal Z_{n-1,n}$ be a positive element
such that $\mathrm{rank}(e(t))=n$ for all $t\in (0,1]$ and $\mathrm{rank}(e(0))=n-1$ (so that
$(n-1)[e]\leq [1]\leq n[e]$ in the Cuntz semigroup of $\mathcal Z_{n-1,n}$). In the proof of Lemma 5.4 of 
\cite{kerdet} a selfadjoint  element $h\in \her(e)$ is constructed such that $h\sim_{\Tr}1$ (in $\mathcal Z_{n-1,n}$) and 
$e^{ih}=(u,v)$ for some unitaries  $u,v\in \mathrm U^0(\her(e))$. Moving these elements with the homomorphism
$\phi$, we get $e'\in B_+$, $h'\in \her(e')_{\sa}$, and $u',v'\in U^0(\her(e'))$, 
such that $(n-1)[e']\leq [p]\leq n[e']$, $h'\sim_{\Tr}p$, and $e^{i h'}=(u',v')$.
Now choose $n>2m$. Then $(2m+1)[e']\leq [p]\leq m[1]$, where $1$ is the unit of $A$. By Lemma \ref{strongcomp},
this implies that $e'\precsim 1$; i.e., there exists $x\in M_m(A)$ such that $x^*x=e'$ and $xx^*\in A$.
Let $x=w|x|$ be the polar decomposition of $x$ (in $M_n(A)^{**}$). Then the selfadjoint $h''=wh'w^*$, and the unitaries
$u''=wuw^*$ and $v''=wvw^*$ have the desired properties.
\end{proof}

\begin{lemma}\label{reducetoexp}
Let $A$ be a pure unital  $\Cstar$-algebra. Let $u\in \mathrm U^0(A)$ be such that $\Delta_{\Tr}(u)=0$. Then 
$u=\prod_{j=1}^M (u_j,v_j)\cdot e^{ih}$ for some $u_1,v_1,\dots, u_M,v_M\in \mathrm U^0(A)$ and  some 
$h\in A_{\sa}$ such that $h\sim_{\Tr}0$.
\end{lemma}
\begin{proof}
Since $u\in \mathrm U^0(A)$ we have   $u=\prod_{j=1}^n e^{ih_j}$, where $h_1,\dots,h_n\in A_{\sa}$, and since $\Delta_{\Tr}(u)=0$ we also have 
$\sum_{j=1}^n h_j\sim_{\Tr}p-q$ for some projections $p,q\in M_m(A)$. 
Applying the previous lemma, we can write 
\[
u=(u',v')(u'',v'')\prod_{j=1}^{n+2}e^{ih_j},
\] 
where now $\sum_{j=1}^{n+2} h_j\sim_{\Tr}0$.
It thus suffices to prove the lemma for the unitary  $\prod_{j=1}^{n+2}e^{ih_j}$.
Let $N \in\N$. Then 
  \begin{align*}
 \prod_{j=1}^{n+2}e^{ih_j}&=\prod_{j=1}^{n+2} (e^{ih_j/N})^N\\
  &=\prod_{j=1}^M (u_j,v_j) \cdot \Big(\prod_{j=1}^{n+2} e^{ih_j/N}\Big)^N.
  \end{align*}
  Here the commutators $(u_j,v_j)$ result simply from rearranging the factors of 
  the first product. In particular, $u_j,v_j\in \mathrm U^0(A)$ for all $j$. We can choose $N$ large enough so that   
  $\prod_{j=1}^{n+2} e^{ih_j/N}=e^{ih}$, for some $h\in A_{\sa}$. By \cite[Lemma 3(b)]{dlHarpe-Skandalis1}, 
  the trace of the logarithm of a product of $n+2$ unitaries belonging to a sufficiently small neighborhood of the identity is equal to the sum of the trace of the logarithm of each of the unitaries. 
Thus,  
  for $N$ large enough, we have with $h\sim_{\Tr}\sum_{j=1}^{n+2} h_j/N\sim_{\Tr}0$. Therefore,
  \[
  u=\prod_{j=1}^M (u_j,v_j)\cdot e^{iNh},
  \]
  with $h\in A_{\sa}$ such that $h\sim_{\Tr}0$. This proves the lemma
\end{proof}

\begin{lemma}
Let $m\in \N$, $R>0$, and $\epsilon>0$. Then there exists $M\in \N$  with the following property: If $A$ is a $\Cstar$-algebra and  $a_1,\dots,a_m\in A_{\sa}$ are such that $\|a_i\|\leq R$ for all $i$, then
there exist $x_1,y_1,\dots,x_{M},y_M\in A_{\sa}$ and  $c\in A_\sa$ such that 
\begin{align*}
e^{i(a_1+\dots+a_m)} &=\prod_{k=1}^M (e^{ix_k},e^{iy_k})\cdot  e^{ia_1}\cdots e^{ia_m}\cdot e^{ic},\\
\|x_k\|,\|y_k\| &\leq \epsilon \cdot \sum_{j=1}^m\|a_j\|, \hbox{ for }k=1,\dots,M,\\
\|c\| &\leq \epsilon \cdot \sum_{j=1}^m\|a_j\|,\hbox{ and }c \sim_{\Tr}0.
\end{align*}
\end{lemma}
\begin{proof}
By \cite[Theorem 2]{suzuki}, for $\lambda\in \R$ such that $|\lambda|<\frac{1}{mR}\cdot (\ln 2-\frac 1 2)$ we have
\[
e^{i\lambda a_1+\cdots+i\lambda a_m}=e^{i\lambda a_1}e^{i\lambda a_2}\cdots 
e^{i\lambda a_m}\cdot e^{ic(\lambda)},
\]
where $\|c(\lambda)\|\leq L\lambda^2\max_j \|a_j\|^2$ and the constant $L>0$ is dependent on $m$ and $R$ only. 
Furthermore, by \cite[Lemma 3(b)]{dlHarpe-Skandalis1}, for $|\lambda|$ small enough (depending only on $m$ and $R$), the trace of the logarithm of the right side is equal to $\sum_{j=1}^mi\lambda \Tr(a_j)+\Tr(c(\lambda))$; whence $c(\lambda)\sim_{\Tr}0$. 
 Now let us choose $\lambda=\frac 1 N$, with $N\in \N$. Then for $N$ large enough (depending only on $m$   and $R$) we have that
\begin{align}\label{zassenhaus}
e^{ia_1/N+\cdots +ia_m/N}=e^{ia_1/N}e^{ia_2/N}\cdots e^{ia_m/N}\cdot e^{c_N},
\end{align}
where  $c_N\in A_\sa$ satisfies that $c_N\sim_{\Tr}0$ and  
\[
\|c_N\|\leq \frac{L}{N^2}\cdot \max_i \|a_i\|^2\leq \frac{LR}{N^2}\cdot \max_i \|a_i\|.
\]  
Raising to the $N$ on both sides of \eqref{zassenhaus} we get 
\begin{align*}
e^{ia_1+\dots+ia_m} &=(e^{a_1/N}e^{a_2/N}\cdots e^{a_m/N}e^{c_N})^N\\
& =\prod_{k=1}^M(e^{ix_k},e^{iy_k})\cdot e^{a_1}e^{a_2}\cdots e^{a_m}e^{Nc_N}.
\end{align*}
The commutators $(e^{ix_k},e^{iy_k})$ in the last expression result from rearranging the terms in  $(e^{a_1/N}e^{a_2/N}\cdots e^{a_m/N}\cdot e^{c})^N$. Notice that $M$ depends only  on $N$ and $m$. Choosing
$N>\frac 1 \epsilon$ we arrange for $\|x_k\|,\|y_k\|\leq \epsilon \max_j \|a_j\|$. Choosing  $N>\frac{LR}{\epsilon}$ 
we also get that  $\|Nc_N\|\leq \frac{LR}{N}\cdot \max_i \|a_i\|\leq \epsilon \max_j \|a_j\|$. 
\end{proof}

\begin{proposition}\label{exphalvenorm}
There exists $N\in \N$ such that the following holds:
If  $A=M_2(B)$, where $B$ is a  pure $\Cstar$-algebra with compact $\mathrm{Prim}(B)$ and whose bounded 2-quasitraces are traces,
and  $h\in A_\sa$ is such that $h\sim_\Tr 0$ and  $\|h\|\leq 1$,  then 
\[
e^{ih}=\prod_{j=1}^N (u_j,v_j)\cdot e^{ic}
\] 
for some $c\in A_\sa$ and some $u_1,v_1,\dots,u_N,v_N\in \mathrm U^0(A)$ such that 
\begin{gather*}
c\sim_{\Tr}0, \, \|c\|\leq \frac 1 2\|h\|, \hbox{ and}\\ 
\|u_j-1\|,\|v_j-1\|\leq \|e^{ih}-1\|^{\frac 1 2}, \hbox{ for all $j=1,\dots,N$.}
\end{gather*}
\end{proposition}
\begin{proof}
Let $h\in A_\sa$ be such that $h\sim_\Tr 0$ and $\|h\|\leq 1$.
By Theorem \ref{reductions}, $h$ is a sum of $7$ commutators, and by Lemma \ref{matricesnilpotents} (ii), each of these commutators is a sum of at most 14 nilpotents of order 2. Furthermore, since $h$ is selfadjoint we can assume that 
these nilpotent elements  have the form $[z^*,z]$, with $z^2=0$. We thus have
$
h=\sum_{k=1}^{m} [z_k^*,z_k]$,
for some $z_1,\dots,z_{m}\in A$ such that $z_k^2=0$.  Here  $m=7\times 14$ and  $\|z_k\|^2\leq C'\|h\|$ for all $k$, where $C'$ is a universal constant. However, dividing the elements $z_k$ by a large natural number ($>C'^{\frac 1 2}$) and  enlarging $m$, we can assume instead that
$\|z_k\|^2\leq \|h\|$ for all $k=1,\dots,m$. Let us assume this. Notice that now $\|[z_k,z_k^*]\|\leq \|h\|\leq 1$.
By the previous lemma applied with $\epsilon=\frac 1 2$, $m$, and $R=\frac 1 2$, we have
\begin{align*}
e^{ih} &=e^{i[z_1,z_1^*]+\cdots +i[z_m,z_m^*]}\\
&=\prod_{k=1}^{M} (e^{ix_k},e^{iy_k})\cdot \prod_{k=1}^{m} e^{i[z_k^*,z_k]}\cdot e^{ic},
\end{align*}
where $x_1,y_1,\dots, x_M,y_M\in A$ and $c\in A_\sa$ are such that $c\sim_{\Tr}0$, 
$\|c\|\leq \frac 1 2\|h\|$, and $\|x_k\|,\|y_k\|\leq \frac 1 2 \|h\|$ for all $k$. Notice that $\|e^{ix_k}-1\|, \|e^{iy_k}-1\|\leq \|e^{ih}-1\|^{\frac 1 2}$ for all $k=1,\dots,M$. It remains to show that the terms $e^{i[z_k,z_k^*]}$ are also expressible as commutators. By \cite[Lemma 2.4 (ii)]{kerdet}, for all $z\in A$ such that  $z^2=0$ and $\|z\|^2\leq \frac \pi 2$ we have $e^{i[z^*,z]}=(u,v)$ for some unitaries $u,v\in \mathrm U^0(A)$ such that $\|u-1\|, \|v-1\|\leq \|e^{[z,z^*]-1}\|^{\frac 1 2}$.
 Applying this to each $z_k$, we get   $e^{i[z_k,z_k^*]}=(u_k,v_k)$, where $u_k,v_k\in \mathrm U^0(A)$
are such that 
\[
\|u_k-1\|,\|v_k-1\|\leq\|e^{i[z_k,z_k^*]}-1\|^{\frac 1 2}\leq \|e^{ih}-1\|^{\frac 1 2}
\] 
for all $k=1,\dots,m$.
\end{proof}


\begin{lemma}\label{orthogonalseries}
Let $A$ be a pure $\Cstar$-algebra with strictly positive element $d\in A_+$. Let $\epsilon>0$. Then there exist pairwise orthogonal positive elements such that  
$a_i\sim b_i$  and 
$b_i\precsim a_{i+1}+b_{i+1}$ for all $i$, and $[(d-\epsilon)_+]\leq 11[a_1]$.
\end{lemma}

\begin{proof}
Let $[d]_s\in \Cu(A)$ denote the ``soft" element 
associated to $[d]$. Using that this element is infinitely divisible, we can find $[c_1],[c_2],[c_3],\dots \in \Cu(A)$ such that $[c_i]=2[c_{i+1}]$ for all $i$
and $10[c_1]=[d_s]$.  Let us pick the representatives $c_i$ such that $\|c_i\|\to 0$. Since $[d]\leq 11[c_1]$, there exists $\delta_1>0$ such that $[(d-\epsilon)_+]\leq 11[(c_1-\delta_+1)]$.
Let us continue choosing $\delta_2,\delta_3,\ldots$ such that $(c_i-\delta_i)_+\precsim (c_{i+1}-\delta_{i+1})_+\oplus  (c_{i+1}-\delta_{i+1})_+$	
for all $i$. Consider the element 
\[
c=\bigoplus_{i=1}^\infty (c_i\oplus c_i)\in  (A\otimes \mathcal K)_+.
\]
Notice that $d_\tau(c)\leq \frac{2}{5}d_\tau(d)$ for all $\tau\in \QT(A)$. By Lemma \ref{strongcomp}, $c\precsim d$. Let $v\in \her(d)^{**}$ denote the partial isometry  implementing this subequivalence. Let us define 
\begin{align*}
a_i &=v((c_i-\delta_i)_+\oplus 0)v^*,\\
b_i &=v(0\oplus (c_i-\delta_i)_+)v^*,
\end{align*} 
for all $i$. The elements $a_1,b_1,a_2,b_2,\dots$ are pairwise orthogonal. They satisfy $a_i\sim b_i$ and $b_i\precsim a_{i+1}+b_{i+1}$ for all $i$ and
$[(d-\epsilon)_+]\leq 11[(c_1-\delta_1)+]=11[a_1]$, as desired. 
\end{proof}

\begin{lemma} \label{multiplicativefack}
Let $A$ be a pure $\Cstar$-algebra with $\mathrm{Prim}(A)$ compact and whose bounded 2-quasitraces are traces.
Let $a_1,b_1,a_2,b_2,\dots\in A_+$ be pairwise orthogonal positive elements as in the previous lemma. 
Let $h\in \her(a_1)$ be a selfadjoint element  such that $h\sim_{\Tr}0$. Then $e^{ih}\in \mathrm{DU}^0(A)$.  
\end{lemma}

\begin{proof}
The proof uses the the multiplicative version of ``Fack's technique", as applied in  \cite[Lemma 6.5]{kerdet}. 

Since $e^{ih}=(e^{ih/N})^N$ for all $N\in \N$ we can assume that $\|h\|<\delta$ for any prescribed $\delta$.
Let us choose $\delta$ such that \cite[Proposition 5.18]{dlHarpe-Skandalis2} is applicable to any unitary within a distance of at most 
$\delta$ of 1.

By Proposition \ref{exphalvenorm} applied in $\her(a_1+b_1)$, there exist unitaries
$u_i^{(1)},v_i^{(1)}$ in $\mathrm U^0(\her(a_1+b_1))$, for $i=1,\dots,N$,  such that
\[
e^{ih}=\prod_{i=1}^N (u_i^{(1)},v_i^{(1)})e^{ih_1'},
\]
where $h_1'\in\her(a_1+b_1)_{\sa}$, $h_1'\sim_{\Tr}0$, and $\|h_1\|'<\frac{\|h\|}{2^1}$. Next, 
by \cite[Proposition 5.18]{dlHarpe-Skandalis2} (see also \cite[Lemma 6.4]{kerdet}), there exist unitaries $w^{(1)}_1,x^{(1)}_1,w^{(1)}_2,x^{(1)}_2$ in $\mathrm U^0(\her(a_1+b_1))$ such that 
\[
e^{ih_1'}=(w^{(1)}_1,x^{(1)}_1)(w^{(1)}_2,x^{(1)}_2)e^{ih_1''},
\]
and $h_1''\in \her(b_2)_{\sa}$.  Finally,
 by \cite[Lemma 5.17]{dlHarpe-Skandalis2} applied in $\her(b_1+a_2+b_2)$, we have
$
e^{ih_1''}=(y^{(1)},z^{(1)})e^{ih_2}$,
with $y^{(1)},z^{(1)}\in \mathrm U^0(\her(b_1+a_2+b_2))$ and $h_2\in \her(a_2+b_2)_{\sa}$.
Next, we apply again Proposition \ref{exphalvenorm} in $\her(a_2+b_2)$:
\[
e^{ih_2}=\prod_{i=1}^N (u_i^{(2)},v_i^{(2)})e^{ih_2'},
\]
with $h_2'\sim_{\Tr}0$ and $\|h_2'\|<\frac{1}{2^2}$, followed by  applications of \cite[Proposition 5.18]{dlHarpe-Skandalis2} and \cite[Lemma 5.17]{dlHarpe-Skandalis2}:
\begin{align*}
e^{ih_2'} &=(w^{(2)}_1,x^{(2)}_1)(w^{(2)}_2,x^{(2)}_2)e^{ih_2''},\\
&=(w^{(2)}_1,x^{(2)}_1)(w^{(2)}_2,x^{(2)}_2)(y^{(2)},z^{(2)})e^{ih_3},
\end{align*}
where $h_2''\in \her(b_2)_{\sa}$ and $h_2''\sim_{\Tr}0$, and   $h_3\in \her(a_3+b_3)_{\sa}$ and $h_3\sim_{\Tr}0$. Continuing this strategy we construct, for each $n\in \N$, 
\begin{enumerate}
\item
unitaries $u_1^{(n)},v_1^{(n)},\dots,u_{N}^{(n)},v_{N}^{(n)}$ and $ w^{(n)}_1,x^{(n)}_1,w^{(n)}_2,x^{(n)}_2$ in $U^0(\her(b_n))$, 
\item
unitaries  $y^{(n)},z^{(n)}$ in $\mathrm U^0(\her(b_n+a_{n+1}+b_{n+1})$, and
\item
a selfadjoint $h_n\in \her(e_n)_{\sa}$, 
\end{enumerate}
such that $h_n\sim_{\Tr}0$ and 
\[
e^{ih}=\prod_{k=1}^{n-1}\Big(\prod_{i=1}^N(u_i^{(k)},v_i^{(k)})\Big) (w^{(k)}_1,x^{(k)}_1)(w^{(k)}_1,x^{(k)}_1)(y^{(k)},z^{(k)})\cdot e^{ih_n}.
\]
Notice that $h_n\to 0$. Thus, the above formula yields and expression of $e^{ih}$   as an infinite product of commutators.
By the pairwise orthogonality of the $a_n$s and $b_n$s, we can gather the terms of this infinite product into subsequences, each of them equal to a finite  product of commutators. This is done in the same manner as in the proof of \cite[Proposition 6.1]{dlHarpe-Skandalis2}.  First, we group together the commutators $(y^{(k)},z^{(k)})$ in the product above:
\[
e^{ih}=\prod_{k=1}^{n-1}\Big(\prod_{i=1}^N(\tilde u_i^{(k)},\tilde v_i^{(k)})\Big) (\tilde w^{(k)}_1,\tilde x^{(k)}_1)(\tilde w^{(k)}_1,\tilde x^{(k)}_1)\prod_{k=1}^{n-1}(y^{(k)},z^{(k)})\cdot e^{ih_n},
\]
where 
\begin{align*}
\tilde u_i^{(k)} &=(y^{(k-1)},z^{(k-1)})u_i^{(k)}(y^{(k-1)},z^{(k-1)})^{-1},\\
\tilde v_i^{(k)} &=(y^{(k-1)},z^{(k-1)})v_i^{(k)}(y^{(k-1)},z^{(k-1)})^{-1},\\
\tilde w^{(k)}_j &=(y^{(k-1)},z^{(k-1)})w_j^{(k)}(y^{(k-1)},z^{(k-1)})^{-1},\\
\tilde x^{(k)}_j &=(y^{(k-1)},z^{(k-1)})x_j^{(k)}(y^{(k-1)},z^{(k-1)})^{-1},
\end{align*}
for all $i=1,\dots,N$, $j=1,2$, and $k=2,\dots,n$. Since $(y^{(k)},z^{(k)})-1$ belongs to $\her(b_k)+\her(a_{k+1}+b_{k+1})$, the modified unitaries  
$\tilde u_i^{(k)},\tilde v_i^{(k)},\tilde w_j^{(k)},x_j^{(k)}$ continue to belong to $\mathrm{U}^0(\her(a_k+b_k))$. Therefore,
\begin{enumerate}
\item
$\prod_{k=1}^{\infty}\Big(\prod_{i=1}^N(\tilde u_i^{(k)},\tilde v_i^{(k)})\Big)\cdot (\tilde w^{(k)}_1,\tilde x^{(k)}_1)\cdot (\tilde w^{(k)}_2,\tilde x^{(k)}_2)$ is a product of $N+2$ commutators,
\item
$\prod_{k=1}^\infty (y^{(2k-1)}),z^{(2k-1)})$ is a single commutator,

\item
$\prod_{k=1}^\infty (y^{(2k)}),z^{(2k)})$ is a single commutator.
\end{enumerate}
We thus arrive at an expression of $e^{ih}$ as a product of $N+4$ commutators.
\end{proof}

\begin{proof}[Proof of Theorem \ref{determinant}]
	It is clear that every unitary in $\mathrm{DU}^0(A)$ is in $\ker \Delta_{\Tr}$.
To prove the converse, it suffices, by Lemma \ref{reducetoexp},  to show that $e^{ih}$, with $h\sim_{\Tr}0$ is a finite product of commutators. Writing $e^{ih}=(e^{ih/N})^N$, 
  we can assume that $\|h\|<\delta$ for any prescribed $\delta$.
We will specify how small should $\delta$ be soon. 
By \cite[Proposition 5.18]{dlHarpe-Skandalis2}, $e^{ih}$ is a product of commutators times $e^{ih'}$, with $h'\in \her(e_{1,1})$ and $h'\sim_{\Tr}0$. Let us choose $d\in \her(e_{2,2}+e_{3,3})$ and $\epsilon>0$ such that $N[d]\leq [e_{2,2}+e_{3,3}]\leq (N+1)[(d-2\epsilon)_+]$, for $N$ large enough (how large to be specified soon).
Let us find pairwise orthogonal elements $d_1,\dots d_N\her(e_{2,2}+e_{3,3})$
such that $d_i\sim (d-\epsilon)_+$ for all $i$.
Since $[e_{1,1}]\leq N[(d-2\epsilon)_+]$, 
we have  $e_{1,1}\precsim (\sum_i d_i-\epsilon)_+$. By repeated application of \cite{dlHarpe-Skandalis2}, we can express $e^{ih'}$
 as a finite product of commutators times $e^{ih''}$, with $h''\in \her((d_1-\epsilon)_+)$ selfadjoint such that $h''\sim_{\Tr}0$ (in $\her((d_1-\epsilon)_+)$).
Let us choose a sequence $a_1,b_1,a_2,b_2\ldots\in \her(d_2+\cdots+d_{13})$ as in Lemma \ref{orthogonalseries}. Then $12[d_1]=[d_2+\cdots+d_{13}]\leq 11[a_1]$.
Thus, $(d_{1}-\epsilon)_+\precsim a_1$. We can therefore express $e^{ih''}$	as a commutator times $e^{ih'''}$, with $h'''\in \her(a_1)$ and $h'''\sim_{\Tr}0$.
	By Lemma \ref{multiplicativefack}, $e^{ih'''}$ is a finite product of commutators.
	  \end{proof}

\end{document}